\documentclass[reqno, 11pt]{amsart} % draft
\usepackage{amsfonts, amsmath, amssymb, amsthm}
\usepackage[margin=2.75cm, heightrounded]{geometry}
\usepackage{fancybox}
\usepackage{hhline, float}
\usepackage{mathrsfs}
\usepackage{nccmath}
\usepackage[dvipsnames]{xcolor}
\usepackage[colorlinks=true]{hyperref}
\hypersetup{citecolor=red, linkcolor=blue}

\newtheorem{thm}{Theorem}[section]
\newtheorem{lemma}[thm]{Lemma}
\newtheorem{prop}[thm]{Proposition}
\newtheorem{cor}[thm]{Corollary}

\newtheorem{thmx}{Theorem}

\theoremstyle{definition}
\newtheorem{rmk}[thm]{Remark}
\newtheorem{defn}[thm]{Definition}

\newcommand{\ga}{\gamma}
\newcommand{\ep}{\epsilon}
\newcommand{\vep}{\varepsilon}
\newcommand{\vrh}{\varrho}

\newcommand{\om}{\omega}
\newcommand{\pa}{\partial}

\newcommand{\N}{\mathbb{N}}
\newcommand{\R}{\mathbb{R}}

\newcommand{\mone}{\mathbf{1}}

\newcommand{\mca}{\mathcal{A}}
\newcommand{\mcd}{\mathcal{D}}
\newcommand{\mcf}{\mathcal{F}}
\newcommand{\mch}{\mathcal{H}}
\newcommand{\mci}{\mathcal{I}}
\newcommand{\mcj}{\mathcal{J}}
\newcommand{\mck}{\mathcal{K}}
\newcommand{\mcr}{\mathcal{R}}
\newcommand{\mcs}{\mathcal{S}}
\newcommand{\mct}{\mathcal{T}}
\newcommand{\mcv}{\mathcal{V}}
\newcommand{\mcw}{\mathcal{W}}
\newcommand{\mcz}{\mathcal{Z}}

\newcommand{\mfy}{\mathfrak{y}}

\newcommand{\msq}{\mathscr{Q}}
\newcommand{\msr}{\mathscr{R}}

\newcommand{\hf}{\hat{f}}
\newcommand{\hvrh}{\hat{\varrho}}

\newcommand{\tit}{\tilde{t}}
\newcommand{\ty}{\tilde{y}}
\newcommand{\tz}{\tilde{z}}

\newcommand{\tla}{\tilde{\lambda}}
\newcommand{\tvrh}{\tilde{\varrho}}

\newcommand{\tom}{\tilde{\omega}}
\newcommand{\wtmcz}{\widetilde{\mcz}}

\newcommand{\lsr}{\lfloor s \rfloor}

\newcommand{\supp}{\textup{supp}}
\newcommand{\tin}{\textup{in}}
\newcommand{\tout}{\textup{out}}
\newcommand{\ext}{\textup{Ext}}
\newcommand{\core}{\textup{Core}}
\newcommand{\neck}{\textup{Neck}}

\renewcommand{\(}{\left(}
\renewcommand{\)}{\right)}
\newcommand{\la}{\left\langle}
\newcommand{\ra}{\right\rangle}

\newcommand{\dx}{\textup{d}x}
\newcommand{\dy}{\textup{d}y}
\newcommand{\dom}{\textup{d}\omega}

\allowdisplaybreaks
\numberwithin{equation}{section}

\makeatletter
\@namedef{subjclassname@2020}{%
\textup{2020} Mathematics Subject Classification}
\makeatother

\begin{document}
\title[Sharp stability estimates for fractional Sobolev inequalities]{Sharp quantitative stability estimates for \\ critical points of fractional Sobolev inequalities}

\author{Haixia Chen}
\address[Haixia Chen]{Department of Mathematics and Research Institute for Natural Sciences, College of Natural Sciences, Hanyang University, 222 Wangsimni-ro Seongdong-gu, Seoul 04763, Republic of Korea}
\email{hxchen29@hanyang.ac.kr chenhaixia157@gmail.com}

\author{Seunghyeok Kim}
\address[Seunghyeok Kim]{Department of Mathematics and Research Institute for Natural Sciences, College of Natural Sciences, Hanyang University, 222 Wangsimni-ro Seongdong-gu, Seoul 04763, Republic of Korea,
School of Mathematics, Korea Institute for Advanced Study, 85 Hoegiro Dongdaemun-gu, Seoul 02455, Republic of Korea.}
\email{shkim0401@hanyang.ac.kr shkim0401@gmail.com}

\author{Juncheng Wei}
\address[Juncheng Wei] {Department of Mathematics, Chinese University of Hong Kong, Shatin, NT, Hong Kong}
\email{wei@math.cuhk.edu.hk}

\begin{abstract}
By developing a unified approach based on integral representations, we establish sharp quantitative stability estimates for critical points
of the fractional Sobolev inequalities induced by the embedding $\dot{H}^s(\R^n) \hookrightarrow L^{2n \over n-2s}(\R^n)$ in the whole range of $s \in (0,\frac{n}{2})$.
\end{abstract}

\date{\today}
\thanks{S. Kim was supported by Basic Science Research Program
through the National Research Foundation of Korea (NRF) funded by the Ministry of Science and ICT (NRF2020R1C1C1A01010133, NRF2020R1A4A3079066).
The research of J. Wei is partially supported by GRF grant entitled "New frontiers in singular limits of elliptic and parabolic equations".
Part of the work was done when Kim visited CUHK. He wants to express gratitude for its hospitality.}
\maketitle

\section{Introduction}
Given $n \in \N$ and $s \in \R$, let $\dot{H}^s(\R^n)$ be the homogeneous Sobolev space of fractional order $s$ defined as
\[\dot{H}^s(\R^n) := \left\{u \in \mcs'(\R^n): \mcf u \in L^1_{\text{loc}}(\R^n),\, \|u\|_{\dot{H}^s(\R^n)} := \(\int_{\R^n} |\xi|^{2s} |\mcf u(\xi)|^2 \text{d}\xi\)^{\frac{1}{2}} < \infty\right\}\]
where $\mcf u$ is the Fourier transform of $u$, and $\mcs'(\R^n)$ is the space of tempered distributions, i.e., the continuous dual space of the Schwartz space $\mcs(\R^n)$.
As shown in \cite{BCD}, $\dot{H}^s(\R^n)$ is a Hilbert space if and only if $s < \frac{n}{2}$. Moreover, if $u \in \mcs(\R^n)$, then
\[\|u\|_{\dot{H}^s(\R^n)}^2 = \left\|(-\Delta)^{s \over 2} u\right\|_{L^2(\R^n)}^2 = \int_{\R^n} u (-\Delta)^s u \quad \text{where } \mcf \((-\Delta)^s u\)(\xi) := |\xi|^{2s} \hat{u}(\xi).\]
The space $\dot{H}^s(\R^n)$ with $s < \frac{n}{2}$ is realized as the completion of $\mcs(\R^n)$ under the norm $\|\cdot\|_{\dot{H}^s(\R^n)}$.

\medskip
For any $s \in (0,\frac{n}{2})$, there is an optimal constant $S_{n,s} > 0$ depending only on $n$ and $s$ such that
\begin{equation}\label{eq:Sobolev}
S_{n,s} \|u\|_{L^{p+1}(\R^n)} \le \|u\|_{\dot{H}^s(\R^n)} \quad \text{for all } u \in \dot{H}^s(\R^n) \quad \text{where } p := \frac{n+2s}{n-2s},
\end{equation}
referred to as the fractional Sobolev inequality.
Lieb \cite{Lie} proved that the set of the extremizers of \eqref{eq:Sobolev} consists of non-zero constant multiples of the functions (often called the bubbles)
\begin{equation}\label{eq:bubble}
U[z,\lambda](x) = \alpha_{n,s} \(\frac{\lambda}{1+\lambda^2|x-z|^2}\)^{n-2s \over 2} \quad \text{for } x \in \R^n
\end{equation}
where $\alpha_{n,s} := 2^{n-2s \over 2} [\Gamma({n+2s \over 2})/\Gamma({n-2s \over 2})]^{n-2s \over 4s}$.

According to the standard theory of calculus of variations, an extremizer of \eqref{eq:Sobolev} always solves
\begin{equation}\label{eq:bubbleeq}
(-\Delta)^s u = \mu |u|^{p-1} u \quad \text{in } \R^n, \quad u \in \dot{H}^s(\R^n)
\end{equation}
where $\mu \in \R$ is a Lagrange multiplier. Chen et al. \cite{CLO} classified all positive solutions to \eqref{eq:bubbleeq}, showing that they must assume the form in \eqref{eq:bubble} up to a constant multiple.
Furthermore, D\'avila et al. \cite{DdPS} deduced that if $s \in (0,1)$, then the solution space of a linearized equation of \eqref{eq:bubbleeq}
\begin{equation}\label{eq:nondeg}
(-\Delta)^s Z - p\,U[z,\lambda]^{p-1} Z = 0 \quad \text{in } \R^n, \quad
Z \in L^{\infty}(\R^n).\ \footnotemark
\end{equation}
\footnotetext{More precisely, \eqref{eq:nondeg} is understood as the corresponding integral equation $Z = \Phi_{n,s} \ast \(p\,U[z,\lambda]^{p-1} Z\)$ in $\R^n$ where $\Phi_{n,s}$ is the Riesz potential in \eqref{eq:Riesz}.}%
is spanned by \[Z^a[z,\lambda] = \left. \frac{1}{\lambda} \frac{\pa U[\bar{z},\lambda]}{\pa \bar{z}^a} \right|_{\bar{z}=z} \quad \text{for } a = 1, \ldots, n \quad \text{and} \quad
Z^{n+1}[z,\lambda] = \lambda \left. \frac{\pa U[z,\bar{\lambda}]}{\pa\bar{\lambda}} \right|_{\bar{\lambda}=\lambda},\]
where $\bar{z} = (\bar{z}_1, \ldots, \bar{z}_n) \in \R^n$.
In \cite[Lemma 5.1]{LX}, Li and Xiong extended this non-degeneracy theorem to all $s \in (0,\frac{n}{2})$.
The condition $Z \in L^{\infty}(\R^n)$ in \cite[Lemma 5.1]{LX} can be replaced with $Z \in \dot{H}^s(\R^n)$, which is the content of Lemma \ref{lemma:nondeg} below.

For a further understanding of \eqref{eq:Sobolev}, one can naturally consider its quantitative stability, as proposed by Brezis and Lieb \cite{BL}.
Bianchi and Egnell \cite{BE} proved the existence of a constant $C_{\text{BE}} > 0$ depending only on $n$ such that
\begin{equation}\label{eq:BE}
\inf_{z \in \R^n,\, \lambda > 0,\, c \in \R} \left\|u - cU[z,\lambda]\right\|_{\dot{H}^1(\R^n)}^2
\le C_{\text{BE}} \(\|u\|_{\dot{H}^1(\R^n)}^2 - S_{n,s}^2\|u\|_{L^{p+1}(\R^n)}^2\)
\end{equation}
for any $u \in \dot{H}^1(\R^n)$. Later, their result was generalized by Chen et al. \cite{CFW}, who found a constant $C_{\text{CFW}} > 0$ depending only on $n$ and $s$ such that
\begin{equation}\label{eq:CFW}
\inf_{z \in \R^n,\, \lambda > 0,\, c \in \R} \left\|u - cU[z,\lambda]\right\|_{\dot{H}^s(\R^n)}^2
\le C_{\text{CFW}} \(\|u\|_{\dot{H}^s(\R^n)}^2 - S_{n,s}^2\|u\|_{L^{p+1}(\R^n)}^2\)
\end{equation}
for any $u \in \dot{H}^s(\R^n)$.

Another way to address the stability issue on \eqref{eq:Sobolev} is to consider the qualitative stability for critical points of equation \eqref{eq:bubbleeq}, which is the main objective of this paper.
This problem is difficult because it requires controlling the quantitative behavior of functions with arbitrarily high energy.
The starting point is the following Struwe-type profile decompositions for \eqref{eq:Sobolev} derived by G\'erard \cite[Th\'eor\`eme 1.1]{Ge}.
Refer also to Palatucci and Pisante \cite[Theorem 1.1]{PP} and Fang and Gonz\'alez \cite[Theorem 1.3]{FaG}.
\begin{thmx}\label{thm:Struwe}
Suppose that $n \in \N$, $\nu \in \N$, $s \in (0,\frac{n}{2})$, $p = \frac{n+2s}{n-2s}$, and $S_{n,s} > 0$ is the constant in \eqref{eq:Sobolev}.
Let $\{u_m\}_{m \in \N}$ be a sequence of non-negative functions in $\dot{H}^s(\R^n)$
such that $\(\nu-\frac{1}{2}\) S_{n,s}^n \le \|u_m\|_{\dot{H}^s(\R^n)}^2 \le \(\nu+\frac{1}{2}\) S_{n,s}^n$. If it satisfies
\[\left\|(-\Delta)^s u_m + u_m^p\right\|_{\dot{H}^{-s}(\R^n)} \to 0 \quad \text{as } k \to \infty,\]
then there exist a sequence $\{(z_{1,m}, \ldots, z_{\nu,m})\}_{m \in \N}$ of $\nu$-tuples of points in $\R^n$ and a sequence $\{(\lambda_{1,m}, \ldots, \lambda_{\nu,m})\}_{m \in \N}$ of $\nu$-tuples of positive numbers such that
\[\left\|u_m - \sum_{i=1}^{\nu} U[z_{i,m},\lambda_{i,m}] \right\|_{\dot{H}^s(\R^n)} \to 0 \quad \text{as } m \to \infty.\]
In addition, let $U_{i,m} := U[z_{i,m},\lambda_{i,m}]$ for $i = 1,\ldots,\nu$.
Then there exists $m_0 \in \N$ such that the sequence $\{(U_{1,m},\ldots,U_{\nu,m})\}_{m \ge m_0}$ of $\nu$-tuples of bubbles is $\delta$-interacting in the following sense: If we define the quantity
\begin{equation}\label{eq:qij}
q_{ij} = q(z_i,z_j,\lambda_i,\lambda_j) = \(\frac{\lambda_i}{\lambda_j} + \frac{\lambda_i}{\lambda_j} + \lambda_i\lambda_j |z_i-z_j|^2\)^{-{n-2s \over 2}}
\quad \text{for } i, j = 1,\ldots,\nu,
\end{equation}
then
\[\max_{\substack{i, j = 1,\ldots,\nu, \\ i \ne j}} q\(z_{i,m},z_{j,m},\lambda_{i,m},\lambda_{j,m}\) \le \delta \quad \text{for all } m \ge m_0.\]
\end{thmx}
\noindent If $s = 1$, the above theorem is reduced to one obtained by Struwe \cite{St}.

\medskip
In this paper, we establish sharp quantitative stability estimates of the above decomposition provided any $n \in \N$ and $s \in (0,\frac{n}{2})$.
\begin{thm}\label{thm:main}
Let $n \in \N$, $\nu \in \N$, $s \in (0,\frac{n}{2})$, and $p = \frac{n+2s}{n-2s}$.
There exist a small constant $\delta > 0$ and a large constant $C > 0$ depending only on $n$, $s$, and $\nu$ such that the following statement holds: If $u \in \dot{H}^s(\R^n)$ satisfies
\begin{equation}\label{eq:mainc}
\left\|u - \sum_{i=1}^{\nu} U[\tilde{z}_i, \tilde{\lambda}_i]\right\|_{\dot{H}^s(\R^n)} \le \delta
\end{equation}
for some $\delta$-interacting family $\{U[\tilde{z}_i, \tilde{\lambda}_i]\}_{i=1}^{\nu}$,
then there is a family $\{U[z_i, \lambda_i]\}_{i=1}^{\nu}$ of bubbles such that
\begin{equation}\label{eq:main}
\left\|u - \sum_{i=1}^{\nu} U[z_i,\lambda_i]\right\|_{\dot{H}^s(\R^n)} \le C\begin{cases}
\Gamma(u) &\text{for } \nu = 1,\\
\Gamma(u) &\text{for } 2s < n < 6s \text{ and } \nu \ge 2,\\
\Gamma(u)|\log \Gamma(u)|^{1 \over 2} &\text{for } n = 6s \text{ and } \nu \ge 2,\\
\Gamma(u)^{p \over 2} &\text{for } n > 6s \text{ and } \nu \ge 2
\end{cases}
\end{equation}
where $\Gamma(u) := \|(-\Delta)^su-|u|^{p-1}u\|_{\dot{H}^{-s}(\R^n)}$.

Furthermore, estimate \eqref{eq:main} is sharp for $n > 2s$ and $\nu \ge 2$ in the sense that the power of $\Gamma(u)$ ($|\log \Gamma(u)|$, respectively) cannot be substituted with a larger (smaller, resp.) one.
\end{thm}

As in the proof of \cite[Corollary 3.4]{FG}, we can combine Theorems \ref{thm:Struwe} and \ref{thm:main} to find
\begin{cor}
Let $n \in \N$, $\nu \in \N$, and $s \in (0,\frac{n}{2})$.
For any non-negative function $u \in \dot{H}^s(\R^n)$ such that $\(\nu-\frac{1}{2}\) S_{n,s}^n \le \|u\|_{\dot{H}^s(\R^n)} \le \(\nu+\frac{1}{2}\) S_{n,s}^n$,
there exist $\nu$ bubbles $\{U[z_i,\lambda_i]\}_{i=1}^{\nu}$ such that \eqref{eq:main} holds.
\end{cor}

\medskip
The quantitative stability for sharp inequalities in analysis and geometry is a fascinating subject that has attracted many researchers for decades.
Brezis and Nirenberg \cite{BN} and Brezis and Lieb \cite{BL} began this research direction, examining the Sobolev embeddings $H^1(\Omega) \hookrightarrow L^{2n \over n-2}(\Omega)$ for bounded domains $\Omega$ in $\R^n$.
Later, Bianchi and Egnell \cite{BE} obtained the optimal solution for the embedding $\dot{H}^1(\R^n) \hookrightarrow L^{2n \over n-2}(\R^n)$.
After these seminal works, numerous results of a similar nature appeared in the literature, and the following represents only a fraction of them;
for the Sobolev inequalities $\dot{W}^{1,p}(\R^n) \hookrightarrow L^{np \over n-p}(\R^n)$ in the non-Hilbert setting (for $p \ne 2$) \cite{CFMP, FN, FZ},
the fractional Sobolev inequalities and the Hardy-Littlewood-Sobolev (HLS) inequalities \cite{CFW, DE},
the conformally invariant Sobolev inequalities on Riemannian manifolds \cite{ENS, Fr},
the isoperimetric inequalities \cite{FuMP, FMP, CL, CES, CGPRS},
and so on.
Besides, the smallest possible constants $C_{\text{BE}},\, C_{\text{CFW}} > 0$ in \eqref{eq:BE}--\eqref{eq:CFW} were estimated in \cite{DEFFL, Ko, Ko2, CLT, CLT2}.

In contrast, the quantitative stability for critical points of the Euler-Lagrange equations induced by sharp inequalities has been less understood.
Nonetheless, thanks to the recent works \cite{CFM, FG, DSW}, it was completely analyzed when considering the Sobolev inequality $\dot{H}^1(\R^n) \hookrightarrow L^{2n \over n-2}(\R^n)$:
Ciraolo et al. \cite{CFM} studied the one-bubble case ($\nu = 1$) for $n \ge 3$, Figalli and Glaudo \cite{FG} did the multi-bubble case ($\nu \ge 2$) for $n = 3, 4, 5$, and Deng et al. \cite{DSW} did the multi-bubble case for $n \ge 6$.
In related research, de Nitti and K\"onig \cite{DK} estimated the smallest possible constant $C > 0$ in \eqref{eq:main} for $n \in \N$, $s \in (0,\frac{n}{2})$, and $\nu = 1$.
Additionally, Aryan \cite{Ar} deduced the stability result for the Euler-Lagrange equations of the fractional Sobolev inequalities $\dot{H}^s(\R^n) \hookrightarrow L^{2n \over n-2s}(\R^n)$ with $s \in (0,1)$.
Analogous results for other inequalities can be found in, e.g., \cite{WW, WW2, BGKM2, LZZ, PYZ, CK}.
In this paper, we treat the fractional Sobolev inequalities for all $n \in \N$, $s \in (0,\frac{n}{2})$, and $\nu \in \N$, thereby fully extending all the previous results \cite{CFM, FG, DSW, Ar}.

\medskip
While numerous results in the literature investigated the existence and qualitative behavior of solutions to fractional elliptic problems $(-\Delta)^s u = f(u)$ when $s \in (0,1)$ and $f: \R \to \R$ is a certain function,
studying the case $s \in (1,\frac{n}{2})$ is still at the beginning stage.
Refer to a few works such as \cite{CLO, JLX, JX, CDQ, AGHW, DK, Ko, Ko2, NTZ}.
In fact, a study for the operator $(-\Delta)^s$ for $s > 1$ is an interesting research topic per se; refer to the extension results in \cite{Ya, Ca, CM}, a recent survey paper of Abatangelo \cite{Ab}, and references therein.
We believe that our results may facilitate further researches on higher-order local and non-local elliptic problems.

\medskip \noindent \textbf{Novelty of the proof.}
Here, we will briefly explain the new features of our proof of Theorem \ref{thm:main}. They mainly originate from the fact that we allow $s > 1$.

\medskip \noindent (1) Our method, primarily based on \cite{DSW}, offers a \textit{unified} approach for any choice of $n \in \N$ and $s \in (0,\frac{n}{2})$.
The choice of the norms with which we work depends on $n$:
In what follows, we say that the dimension $n$ is high if $n \ge 6s$ and low if $2s < n < 6s$.
For the high-dimensional case, we use weighted $L^{\infty}(\R^n)$-type norms; refer to Definition \ref{defn:normshigh}.
For the low-dimensional case, we utilize the standard $\dot{H}^s(\R^n)$-norm and $L^{2n \over n+2s}(\R^n)$-norm; see Definition \ref{defn:normslow}.

\medskip \noindent (2) In Proposition \ref{prop:spec}, we derive a spectral inequality that holds for all $s \in (0,\frac{n}{2})$ and $\delta$-interacting families with $\delta > 0$ small.
We do not use bump functions that appeared in the proof of Figalli and Glaudo \cite{FG} and Aryan \cite{Ar} for the case $s \in (0,1]$, resulting in a simpler proof.\footnote{Applying
the idea developed in this paper, the first two authors obtained an analogous inequality in the setting of the Yamabe problem on compact Riemannian manifolds in \cite{CK}.}
Some key ingredients are the fractional Leibniz rule \cite{GO} and Li's Kenig-Ponce-Vega estimate \cite{Li}.

\medskip \noindent (3) Proving Proposition \ref{prop:31} (linear theory) is one of the most delicate parts of the paper.

\begin{itemize}
\item[-] We need to first ensure that the $*$-norm of $f$ is finite when the $**$-norm of $h$ is finite, because our domain $\R^n$ is unbounded.
    We will deduce the result by repeatedly applyng the integral representation of $f$ and the HLS inequality.
\item[-] When $s \in (0,1]$, one can employ the barrier argument based on the maximum principle for narrow domains to control $f$ in the neck region, as described in \cite{DSW, Ar}.
    However, extending this approach to large $s > 1$ is extremely challenging.
    In this study, we introduce a totally different method based on potential analysis, which simplifies the argument and allows us to treat the case $s \in (0,\frac{n}{2})$ simultaneously.
\item[-] To estimate $f$ in the core region, we require a H\"older continuity for the rescaled function $\hf$, as mentioned in Lemma \ref{lemma:3c6} and \eqref{eq:3c35}.
    While the standard theory of elliptic regularity is applicable for $s \in (0,1]$ or $s \in (1,\frac{n}{2}) \cap \N$, it is not yet available for $s \in (1,\frac{n}{2}) \setminus \N$.
    We will directly analyze the representation of $\hf$ to ensure that it has H\"older continuity or even higher-order differentiability.
\item[-] To deduce the limit equation \eqref{eq:3c37}, we have to analyze integrals on $\R^n$.
    We need to divide $\R^n$ into three distinct parts: the singular part, the uniformly convergent part, and the exterior part.
    Although this approach is relatively standard, the interaction between different bubbles necessitates a more refined analysis; see Appendix \ref{subsec:app23} for further details.
    The same strategy can be applied in the derivation of \eqref{eq:last}; see Appendix \ref{apsu3}.
\end{itemize}

\medskip \noindent (4) In Appendix \ref{sec:app}, we prove the non-degeneracy of the bubble in $\dot{H}^s(\R^n)$ and the removability of singularities for nonlocal equations.
They hold for all $s \in (0,\frac{n}{2})$, are of independent interest, and may also be helpful in other contexts.

\medskip \noindent \textbf{Organization of the paper.}
In Section \ref{sec:spec}, we derive a spectral inequality for $\delta$-interacting families.
In Sections \ref{sec:dimhigh1}--\ref{sec:dimlow}, we prove Theorem \ref{thm:main}:
The cases $n \ge 6s$ and $2s < n < 6s$ are treated in Sections \ref{sec:dimhigh1}--\ref{sec:dimhigh2} and \ref{sec:dimlow}, respectively.
In Appendix \ref{sec:app}, we obtain the auxiliary results described in (4) above.
In Appendix \ref{sec:app2}, we carry out technical computations needed in the proof.

We are mainly concerned with the multi-bubble case $\nu \ge 2$, as the single-bubble case $\nu = 1$ and $s \in (0,\frac{n}{2})$ has already treated in \cite{DK}.
For the sake of brevity, we often omit proofs if there is a suitable reference to quote.
In particular, we borrow several estimates obtained in \cite{DSW} for $s = 1$, whenever similar estimates hold for all $s \in (0,\frac{n}{2})$.

\medskip \noindent \textbf{Notations.}
We collect some notations used in the paper.

\medskip \noindent - Let $\lsr$ be the greatest integer that does not exceed $s$.

\medskip \noindent - Let (A) be a condition. We set $\mone_{\text{(A)}} = 1$ if (A) holds and $0$ otherwise.

\medskip \noindent - For $x \in \R^n$ and $r > 0$, we write $B(x,r) = \{\om \in \R^n: |\om-x| < r\}$ and $B(x,r)^c = \{\om \in \R^n: |\om-x| \ge r\}$.

\medskip \noindent - Given $n \in \N$ and $s \in (0,\frac{n}{2})$, let $\Phi_{n,s}$ be the Riesz potential
\begin{equation}\label{eq:Riesz}
\Phi_{n,s}(x) = \frac{\ga_{n,s}}{|x|^{n-2s}} \quad \text{for } x \in \R^n \setminus \{0\} \quad \text{where } \ga_{n,s} := \frac{\Gamma\(\frac{n-2s}{2}\)}{\pi^{n/2}2^{2s}\Gamma(s)}.
\end{equation}

\medskip \noindent - Given a function $u \in \dot{H}^s(\R^n)$, let $\mcf u$ be the Fourier transform of $u$.
The notation $\hat{u}$ is reserved for other use, e.g., a suitable rescaling of $u$.

\medskip \noindent - We use the Japanese bracket notation $\la x \ra = \sqrt{1+|x|^2}$ for $x \in \R^n$.

\medskip \noindent - Unless otherwise stated, $C > 0$ is a universal constant that may vary from line to line and even in the same line.
We write $a_1 \lesssim a_2$ if $a_1 \le Ca_2$, $a_1 \gtrsim a_2$ if $a_1 \ge Ca_2$, and $a_1 \simeq a_2$ if $a_1 \lesssim a_2$ and $a_1 \gtrsim a_2$.

\section{The spectral inequality}\label{sec:spec}
As a preparation step for the proof of Theorem \ref{thm:main}, we derive a spectral inequality \eqref{eq:342} which will be employed in the proof of Propositions \ref{prop:34} and \ref{prop:54}.
It was deduced in \cite[Proposition 3.1]{Ba} and \cite[Proposition 3.10]{FG} when $s = 1$, and in \cite[Lemma 2.5]{Ar} when $s \in (0,1)$.
Here, we present a proof based on a blow-up argument.
\begin{defn}
We write $U_i = U[z_i,\lambda_i]$ for $i = 1,\ldots,\nu$. For $\nu \ge 2$, let $q_{ij}$ be the quantity in \eqref{eq:qij} and
\begin{equation}\label{eq:Q}
\msq = \max\{q_{ij}: i, j = 1,\ldots,\nu,\, i \ne j\}
\end{equation}
so that the $\nu$-tuple $(U_1,\ldots,U_{\nu})$ of bubbles is $\delta$-interacting if and only if $\msq \le \delta$. We also set
\begin{equation}\label{eq:Rij}
\msr_{ij} = \max\left\{\sqrt{\lambda_i \over \lambda_j},\, \sqrt{\lambda_j \over \lambda_i},\,
\sqrt{\lambda_i\lambda_j} |z_i-z_j| \right\} \simeq q_{ij}^{-\frac{1}{n-2s}} \quad \text{for } i, j = 1,\ldots,\nu,\, i \ne j
\end{equation}
and
\begin{equation}\label{eq:R}
\msr = \frac{1}{2} \min\{\msr_{ij}: i, j = 1,\ldots,\nu,\, i \ne j\} \simeq \msq^{-\frac{1}{n-2s}}.
\end{equation}
\end{defn}

\begin{prop}\label{prop:spec}
Let $n \in \N$, $\nu \in \N$, $s \in (0,\frac{n}{2})$, and $\delta_0 > 0$ is sufficiently small.
Suppose that the $\nu$-tuple $(U_1,\ldots,U_{\nu})$ of bubbles is $\delta'$-interacting for some $\delta' \in (0,\delta_0)$.
If $\vrh = \vrh(x)$ is a function in $\dot{H}^s(\R^n)$ that satisfies
\[\int_{\R^n} (-\Delta)^{s \over 2} \vrh \, (-\Delta)^{s \over 2} U_i \dx = \int_{\R^n} (-\Delta)^{s \over 2} \vrh \, (-\Delta)^{s \over 2} Z_i^a \dx = 0\]
for all $i = 1,\ldots,\nu$ and $a = 1,\ldots,n+1$, then there exists a constant $c_0 \in (0,1)$ such that
\begin{equation}\label{eq:342}
\int_{\R^n} \sigma^{p-1} \vrh^2 \dx \le \frac{c_0}{p} \|\vrh\|_{\dot{H}^s(\R^n)}^2
\end{equation}
where $\sigma = \sum_{i=1}^{\nu} U_i$.
\end{prop}
\begin{proof}
The case $\nu = 1$ is clear. In the sequel, we assume that $\nu \ge 2$.

\medskip
To the contrary, suppose that there exist sequences of small positive numbers $\{\delta'_m\}_{m \in \N}$,
$\delta'_m$-interacting $\nu$-tuples of bubbles $\{(U_{1,m},\ldots,U_{\nu,m})\}_{m \in \N}$,
functions $\{\vrh_m\}_{m \in \N}$ in $\dot{H}^s(\R^n)$, and numbers $\{c_m\}_{m \in \N}$ in $(0,1]$ such that $\delta'_m \to 0$ and $c_m \to 1$ as $m \to \infty$,
\begin{equation}\label{eq:343}
\|\vrh_m\|_{\dot{H}^s(\R^n)} = 1, \quad
\int_{\R^n} \sigma_m^{p-1} \vrh_m^2 \dx = \sup \left\{\int_{\R^n} \sigma_m^{p-1} \vrh^2 \dx: \|\vrh\|_{\dot{H}^s(\R^n)}=1 \right\} \ge \frac{c_m}{p},
\end{equation}
and
\begin{equation}\label{eq:344}
\int_{\R^n} U_{i,m}^p \vrh_m \dx = \int_{\R^n} U_{i,m}^{p-1} Z_{i,m}^a \vrh_m \dx = 0
\end{equation}
for $m \in \N$, $i = 1,\ldots,\nu$, and $a = 1,\ldots,n+1$. Here, $\sigma_m := \sum_{i=1}^{\nu} U_{i,m}$ and $Z_{i,m}^a := Z^a[z_{i,m},\lambda_{i,m}]$.
In view of \eqref{eq:343}--\eqref{eq:344}, we know that
\begin{equation}\label{eq:345}
(-\Delta)^s \vrh_m - \mu_m \sigma_m^{p-1} \vrh_m = \sum_{i=1}^{\nu} \mu_{i,m} U_{i,m}^p + \sum_{i=1}^{\nu} \sum_{a=1}^{n+1} \mu_{i,m}^a U_{i,m}^{p-1}Z_{i,m}^a \quad \text{in } \R^n
\end{equation}
where $\mu_m, \mu_{i,m}, \mu_{i,m}^a \in \R$ are Lagrange multipliers. Testing \eqref{eq:345} with $\rho_m$ and using \eqref{eq:344} yield
\begin{equation}\label{eq:346}
\mu_m = \(\int_{\R^n} \sigma_m^{p-1} \vrh_m^2 \dx\)^{-1} \in \left[c(n,s,\nu), c_m^{-1}p\right]
\end{equation}
where the lower bound $c(n,s,\nu)$ is positive and dependent only on $n$, $s$, and $\nu$.
Hence, we may assume that $\mu_m \to \mu_{\infty} \in \left[c(n,s,\nu), p\right]$ as $m \to \infty$.

Let $q_{ij,m}$, $\msq_m$, $\msr_{ij,m}$, and $\msr_m$ be the quantities introduced in \eqref{eq:qij}, \eqref{eq:Q}, \eqref{eq:Rij}, and \eqref{eq:R}, respectively,
where $(z_i, z_j, \lambda_i, \lambda_j)$ is replaced with $(z_{i,m},z_{j,m},\lambda_{i,m},\lambda_{j,m})$.
We present the rest of the proof by dividing it into three steps.

\medskip \noindent \doublebox{\textsc{Step 1.}} We claim that
\begin{equation}\label{eq:347}
\sum_{i=1}^{\nu} \left|\mu_{i,m}\right| + \sum_{i=1}^{\nu} \sum_{a=1}^{n+1} \left|\mu_{i,m}^a\right| \to 0 \quad \text{as } m \to \infty.
\end{equation}
Testing \eqref{eq:345} with $U_{j,m}$ for $j = 1,\ldots,\nu$ and employing \eqref{eq:344}, we obtain
\[- \mu_m \int_{\R^n} \sigma_m^{p-1} U_{j,m} \vrh_m \dx = \sum_{i=1}^{\nu} \mu_{i,m} \int_{\R^n} U_{i,m}^p U_{j,m} \dx
+ \sum_{i=1}^{\nu} \sum_{a=1}^{n+1} \mu_{i,m}^a \int_{\R^n} U_{i,m}^{p-1}Z_{i,m}^a U_{j,m} \dx.\]
On the other hand, \cite[Proposition B.2]{FG} tells us that
\begin{equation}\label{eq:UiUj}
\int_{\R^n} U_i^{\alpha} U_j^{\beta} \dx \simeq q_{ij}^{\min\{\alpha,\beta\}} \quad \text{for any } i, j = 1,\ldots,\nu,\, i \ne j
\end{equation}
provided $\alpha, \beta \ge 0$, $\alpha \ne \beta$, and $\alpha + \beta = p+1$. By \eqref{eq:344} and \eqref{eq:UiUj}, we have
\[\int_{\R^n} U_{i,m}^p U_{j,m} \dx = \begin{cases}
\displaystyle \int_{\R^n} U_{1,0}^{p+1} \dx &\text{if } i = j,\\
O(q_{ij,m}) &\text{if } i \ne j,
\end{cases}\]
\[\int_{\R^n} U_{i,m}^{p-1}Z_{i,m}^a U_{j,m} \dx = \begin{cases}
0 &\text{if } i = j,\\
O(q_{ij,m}) &\text{if } i \ne j,
\end{cases}\]
and
\begin{align*}
\int_{\R^n} \sigma_m^{p-1} U_{j,m} \vrh_m \dx &= \int_{\R^n} \(\sigma_m^{p-1}-U_{j,m}^{p-1}\) U_{j,m} \vrh_m \dx \\
&= O\(\sum_{\substack{i = 1,\ldots,\nu,\\i \ne j}} \left[\left\|U_{i,m}^{p-1}U_{j,m}\right\|_{L^{p+1 \over p}(\R^n)}
+ \left\|U_{i,m}^{p-2}U_{j,m}^2\right\|_{L^{p+1 \over p}(\R^n)} \mone_{n < 6s}\right]\) \\
&= O\(q_{ij,m}^{p-1} + q_{ij,m}^{\min\{p-2,1\}} \mone_{n < 6s}\).
\end{align*}
Thus
\begin{multline}\label{eq:348}
O\(\msq_m^{p-1} + \msq_m^{\min\{p-2,1\}} \mone_{n < 6s}\) = \left[\int_{\R^n} U_{1,0}^{p+1} \dx\right] \mu_j \\
+ \sum_{\substack{i = 1,\ldots,\nu,\\i \ne j}} O(q_{ij,m}) \mu_i + \sum_{\substack{i = 1,\ldots,\nu,\\i \ne j}} \sum_{a=1}^{n+1} O(q_{ij,m}) \mu_i^a.
\end{multline}
Similarly, by testing \eqref{eq:345} with $Z_{j,m}^b$, we get
\begin{multline}\label{eq:349}
O\(\msq_m^{p-1} + \msq_m^{\min\{p-2,1\}} \mone_{n < 6s}\) = \sum_{\substack{i = 1,\ldots,\nu,\\i \ne j}} O(q_{ij,m}) \mu_i \\
+ \sum_{\substack{i = 1,\ldots,\nu,\\i \ne j}} \sum_{a=1}^{n+1} O(q_{ij,m}) \mu_i^a + \left[\int_{\R^n} U_{1,0}^{p-1}\(Z_{1,0}^b\)^2 \dx\right] \mu_j^b.
\end{multline}
Claim \eqref{eq:347} follows from \eqref{eq:348}, \eqref{eq:349}, and the fact that $\msq_m \to 0$ as $m \to \infty$.

\medskip \noindent \doublebox{\textsc{Step 2.}} We verify
\begin{equation}\label{eq:3497}
\lim_{m \to \infty} \int_{B_{i,m}} U_{i,m}^{p-1} \vrh_m^2 \dx = 0 \quad \text{for each } i = 1,\ldots,\nu.
\end{equation}

Let $\chi: \R^n \to [0,1]$ be an arbitrary smooth radial function such that $\chi = 1$ in $B(0,1)$ and $0$ on $B(0,2)^c$.
Also, fixing $i = 1,\ldots,\nu$ and a sequence $\{r_m\}_{m \in \N} \subset (0,\infty)$ of positive numbers such that $r_m \to \infty$ as $m \to \infty$, we set
\[\hvrh_{i,m}(y) = \chi_m(y) \lambda_{i,m}^{-{n-2s \over 2}} \vrh_m\big(\lambda_{i,m}^{-1}y+z_{i,m}\big) \quad \text{for } y \in \R^n \quad \text{where } \chi_m(y) := \chi\(\frac{2y}{r_m}\).\]
By \eqref{eq:345},
\begin{multline}\label{eq:3498}
(-\Delta)^s \hvrh_{i,m} - \mu_m \left\{\lambda_{i,m}^{-{n-2s \over 2}} \sigma_m\big(\lambda_{i,m}^{-1} \cdot + z_{i,m}\big)\right\}^{p-1} \hvrh_{i,m} \\
= \chi_m \lambda_{i,m}^{-{n+2s \over 2}} \left[\sum_{j=1}^{\nu} \mu_{j,m} U_{j,m}^p + \sum_{j=1}^{\nu} \sum_{a=1}^{n+1}
\mu_{j,m}^a U_{j,m}^{p-1}Z_{j,m}^a\right]\big(\lambda_{i,m}^{-1} \cdot + z_{i,m}\big) + \mcr_{i,m} \quad \text{in } \R^n
\end{multline}
where
\[\mcr_{i,m}(y) := \left[(-\Delta)^s \hvrh_{i,m}\right](y) - \chi_m(y) \lambda_{i,m}^{-{n+2s \over 2}} \left[(-\Delta)^s \vrh_m\right]\big(\lambda_{i,m}^{-1}y+z_{i,m}\big) \quad \text{for } y \in \R^n.\]
In addition,
\begin{equation}\label{eq:349c}
\begin{aligned}
\|\hvrh_{i,m}\|_{\dot{H}^s(\R^n)} &\lesssim \left\|(-\Delta)^{s \over 2} \chi_m\right\|_{L^{n \over s}(\R^n)} \|\vrh_m\|_{L^{2n \over n-2s}(\R^n)}
+ \|\chi_m\|_{L^{\infty}(\R^n)} \|\vrh_m\|_{\dot{H}^s(\R^n)} \\
&\lesssim \left\|(-\Delta)^{s \over 2} \chi\right\|_{L^{n \over s}(\R^n)} + 1 \lesssim 1.
\end{aligned}
\end{equation}
Here, we applied a fractional Leibniz rule (see e.g. \cite[Theorem 1]{GO}) for the first inequality, and \eqref{eq:Sobolev} and \eqref{eq:343} for the second inequality.
Also, we employed the Hausdorff-Young inequality, the assumption that $n > 2s$, and $\mcf \chi \in \mcs(\R^n)$ for the last inequality. Therefore, we may assume that
\[\hvrh_{i,m} \rightharpoonup \hvrh_{i,\infty} \quad \text{weakly in } \dot{H}^s(\R^n)
\quad \text{ and } \quad \hvrh_{i,m} \to \hvrh_{i,\infty} \quad \text{a.e.} \quad \text{as } m \to \infty\]
for some $\hvrh_{i,\infty} \in \dot{H}^s(\R^n)$. By \eqref{eq:344},
\begin{equation}\label{eq:349a}
\int_{\R^n} U[0,1]^p \hvrh_{i,\infty} \dy = \int_{\R^n} U[0,1]^{p-1} Z^a[0,1] \hvrh_{i,\infty} \dy = 0 \quad \text{for all } a = 1,\ldots,n+1.
\end{equation}
Let $\psi \in C_c^{\infty}(\R^n)$ be a test function. Then
\begin{align*}
&\begin{medsize}
\displaystyle \ \int_{\R^n} \left\{\lambda_{i,m}^{-{n-2s \over 2}} \sigma_m\big(\lambda_{i,m}^{-1}y+z_{i,m}\big) \right\}^{p-1} \hvrh_{i,m} \psi \dy
\end{medsize} \\
&\begin{medsize}
\displaystyle = \int_{\R^n} U[0,1]^{p-1} \hvrh_{i,m} \psi \dy + O\(\max_{\alpha \in \left\{{p+1 \over p}, {p^2-1 \over p}\right\}}
\sum_{\substack{k = 1,\ldots,\nu, \\ k \ne i}} \left\|\lambda_{i,m}^{-{n-2s \over 2}} U_{k,m}\big(\lambda_{i,m}^{-1}\cdot+z_{i,m}\big)\right\|_{L^{\alpha}\(\supp\, \psi \cap B(0,r_m)\)}^{\alpha p \over p+1}\)
\end{medsize} \\
&\begin{medsize}
\displaystyle \to \int_{\R^n} U[0,1]^{p-1} \hvrh_{i,\infty} \psi \dy %+ 0
\quad \text{as } m \to \infty,
\end{medsize}
\end{align*}
because if we set $z_{ki,m} = \lambda_{k,m}(z_{i,m}-z_{k,m})$, then
\begin{align*}
&\ \({\lambda_{k,m} \over \lambda_{i,m}}\)^{\alpha(n-2s) \over 2} \int_{\supp\, \psi \cap B(0,r_m)}
\frac{\dy}{\la (\lambda_{k,m}/\lambda_{i,m})y + z_{ki,m} \ra^{\alpha(n-2s)}} \\
&= \({\lambda_{k,m} \over \lambda_{i,m}}\)^{{\alpha(n-2s) \over 2}-n} \int_{B(z_{ki,m},\frac{\lambda_{k,m}}{\lambda_{i,m}}r_0)} \frac{dY}{\la Y \ra^{\alpha(n-2s)}} \\
&\lesssim \begin{cases}
\({\lambda_{k,m} \over \lambda_{i,m}}\)^{\alpha(n-2s) \over 2} = o(1) &\begin{medsize}
\displaystyle \text{if } \lim_{m \to \infty} \frac{\lambda_{k,m}}{\lambda_{i,m}} = 0,
\end{medsize} \\
\({\lambda_{k,m} \over \lambda_{i,m}}\)^{{\alpha(n-2s) \over 2}-n} + \({\lambda_{k,m} \over \lambda_{i,m}}\)^{-{\alpha(n-2s) \over 2}} = o(1) &\begin{medsize}
\displaystyle \text{if } \lim_{m \to \infty} \frac{\lambda_{k,m}}{\lambda_{i,m}} = \infty,
\end{medsize} \\
{1 \over \msr_{ki,m}^{\alpha(n-2s)}} = o(1) &\begin{medsize}
\displaystyle \text{if } \lim_{m \to \infty} \frac{\lambda_{k,m}}{\lambda_{i,m}} \in (0,\infty) \text{ (so that } \msr_{ki,m} \simeq |z_{ki,m}|) \end{medsize}
\end{cases}
\end{align*}
for $\alpha \in \{{p+1 \over p}, {p^2-1 \over p}\}$, provided $\supp\, \psi \subset B(0,r_0)$ for some $r_0 > 0$ and $r_m \ge r_0$. Furthermore, H\"older's inequality and \eqref{eq:347} give
\begin{multline*}
\left|\int_{\R^n} \chi_m \lambda_{i,m}^{-{n+2s \over 2}} \left[\sum_{j=1}^{\nu} \mu_{j,m} U_{j,m}^p + \sum_{j=1}^{\nu} \sum_{a=1}^{n+1} \mu_{j,m}^a U_{j,m}^{p-1}Z_{j,m}^a\right]\big(\lambda_{i,m}^{-1}y+z_{i,m}\big) \psi \dy\right| \\
\le \(\sum_{j=1}^{\nu} |\mu_{j,m}| + \sum_{j=1}^{\nu} \sum_{a=1}^{n+1} \left|\mu_{j,m}^a\right|\) \|U[0,1]\|_{L^{2n \over n-2s}(\R^n)}^p \|\psi\|_{L^{2n \over n-2s}(\R^n)} = o(1).
\end{multline*}
Writing
\[\tvrh_{i,m}(y) = \lambda_{i,m}^{-{n-2s \over 2}} \vrh_m\big(\lambda_{i,m}^{-1}y+z_{i,m}\big)
\quad \text{and} \quad \mcf(D^{s,\eta} \psi)(\xi) = i^{-|\eta|} \pa^{\eta}_{\xi} (|\xi|^s) (\mcf \psi)(\xi)\]
where $\eta$ is an $n$-dimensional multi-index, and invoking the generalized Kenig-Ponce-Vega estimate due to Li \cite[Theorem 1.2]{Li}, we deduce
\begin{align*}
\left|\int_{\R^n} \mcr_{i,m} \psi \dy\right| &\le \int_{\R^n} |\tvrh_{i,m}(y)| \left|\left[(-\Delta)^s (\chi_m \psi)\right](y) - \chi_m(y) \left[(-\Delta)^s \psi\right](y)\right| \dy \\
&\le \|\vrh_m\|_{L^{2n \over n-2s}(\R^n)} \|(-\Delta)^s (\chi_m \psi) - \chi_m \left[(-\Delta)^s \psi\right]\|_{L^{n \over s}(\R^n)} \\
&\lesssim \|(-\Delta)^s\chi_m\|_{L^{2n \over s}(\R^n)} \|\psi\|_{L^{2n \over s}(\R^n)} + \sum_{1 \le |\eta| \le 2s} \left\|\pa^{\eta}\chi_m D^{2s,\eta}\psi\right\|_{L^{n \over s}(\R^n)} \\
&\lesssim r_m^{-\frac{3s}{2}} + r_m^{-1} = o(1).
\end{align*}
Here, the empty summation $\sum_{1 \le |\eta| \le 2s}$ is understood as zero for $s \in (0,\frac{1}{2})$.

Accordingly, by taking the limit $m \to \infty$ on \eqref{eq:3498}, we find
\begin{equation}\label{eq:349b}
(-\Delta)^s \hvrh_{i,\infty} - \mu_{\infty} U[0,1]^{p-1} \hvrh_{i,\infty} = 0 \quad \text{in } \R^n
\end{equation}
where $\mu_{\infty} \in \left[c(n,s,\nu), p\right]$. From \eqref{eq:349a}, \eqref{eq:349b}, the fact that $U[0,1]$ is an extremizer of \eqref{eq:Sobolev}, and Lemma \ref{lemma:nondeg},
we conclude that $\hvrh_{i,\infty} = 0$ in $\R^n$. This and \eqref{eq:349c} imply
\[\int_{B_{i,m}} U_{i,m}^{p-1} \vrh_m^2 \dx \le \int_{\R^n} U[0,1]^{p-1} \hvrh_{i,m}^2 \dy \lesssim \(\int_{\R^n} U[0,1]^p \hvrh_{i,m} \dy\)^{p \over p-1} \to 0 \quad \text{as } m \to \infty,\]
which reads \eqref{eq:3497}.

\medskip \noindent \doublebox{\textsc{Step 3.}} Finally, we prove that
\begin{equation}\label{eq:3491}
\lim_{m \to \infty} \int_{\R^n} \sigma_m^{p-1} \vrh_m^2 \dx = 0.
\end{equation}
Its validity will imply that \eqref{eq:342} holds, because it contradicts \eqref{eq:346}.

Given any number $L > 0$, let $B_{i,m} = B(z_{i,m},\frac{L}{\lambda_{i,m}})$ and $B_{i,m}^c$ be its complement. It holds that
\begin{equation}\label{eq:3492}
\int_{B_{i,m}^c} U_{i,m}^{p-1} \vrh_m^2 \dx \le \|U_{i,m}\|_{L^{p+1}(B_{i,m}^c)}^{p-1} \lesssim L^{-2s}
\end{equation}
for $i = 1,\ldots,\nu$. It follows from \eqref{eq:3497} and \eqref{eq:3492} that
\begin{align*}
\int_{\R^n} \sigma_m^{p-1} \vrh_m^2 \dx &\lesssim \sum_{i=1}^{\nu} \int_{\R^n} U_{i,m}^{p-1} \vrh_m^2 \dx \le \sum_{i=1}^{\nu} \left[\int_{B_{i,m}} U_{i,m}^{p-1} \vrh_m^2 \dx + \int_{B_{i,m}^c} U_{i,m}^{p-1} \vrh_m^2 \dx\right] \\
&\lesssim o(1) + L^{-2s},
\end{align*}
which yields \eqref{eq:3491}, because $L > 0$ can be taken arbitrarily large.
\end{proof}

\section{Quantitative stability estimate for dimension $n \ge 6s$ (1)}\label{sec:dimhigh1}
In this section, we establish Theorem \ref{thm:main} assuming that $n \ge 6s$. From now on, we always assume that $\nu \ge 2$.

\medskip
The following two weighted $L^{\infty}(\R^n)$-norms were devised in \cite{DSW} (for the case $s = 1$) to capture the precise pointwise behavior of the bubbles,
which is crucial on determining the optimal exponents of $\Gamma(u)$ in the right-hand side of \eqref{eq:main}.
\begin{defn}\label{defn:normshigh}
Recall the number $\msr > 0$ in \eqref{eq:R} and write $y_i = \lambda_i(x-z_i) \in \R^n$. We define
\[\|h\|_{**} = \sup_{x \in \R^n} |h(x)| \mcv^{-1}(x) \quad \text{and} \quad \|\rho\|_* = \sup_{x \in \R^n} |\rho(x)| \mcw^{-1}(x)\]
with
\begin{equation}\label{eq:mcvw}
\mcv(x) := \sum_{i=1}^{\nu} \(v_i^{\tin} + v_i^{\tout}\)(x) \quad \text{and} \quad \mcw(x) := \sum_{i=1}^{\nu} \(w_i^{\tin} + w_i^{\tout}\)(x)
\end{equation}
where
\begin{equation}\label{eq:vwinout}
\begin{cases}
\displaystyle v_i^{\tin}(x) := \lambda_i^{n+2s \over 2} \frac{\msr^{2s-n}}{\la y_i \ra^{4s}} \mone_{\{|y_i| < \msr\}}(x), \quad
v_i^{\tout}(x) := \lambda_i^{n+2s \over 2} \frac{\msr^{-4s}}{|y_i|^{n-2s}} \mone_{\{|y_i| \ge \msr\}}(x), \\
\displaystyle w_i^{\tin}(x) := \lambda_i^{n-2s \over 2} \frac{\msr^{2s-n}}{\la y_i \ra^{2s}} \mone_{\{|y_i| < \msr\}}(x), \quad
w_i^{\tout}(x) := \lambda_i^{n-2s \over 2} \frac{\msr^{-4s}}{|y_i|^{n-4s}} \mone_{\{|y_i| \ge \msr\}}(x)
\end{cases}
\end{equation}
for $n > 6s$ and $i = 1,\ldots,\nu$, and
\begin{equation}\label{eq:vwinout2}
\begin{cases}
\displaystyle v_i^{\tin}(x) := \lambda_i^{4s} \frac{\msr^{-4s}}{\la y_i \ra^{4s}} \mone_{\{|y_i| < \msr^2\}}(x), \quad
v_i^{\tout}(x) := \lambda_i^{4s} \frac{\msr^{-2s}}{|y_i|^{5s}} \mone_{\{|y_i| \ge \msr^2\}}(x), \\
\displaystyle w_i^{\tin}(x) := \lambda_i^{2s} \frac{\msr^{-4s}}{\la y_i \ra^{2s}} \mone_{\{|y_i| < \msr^2\}}(x), \quad
w_i^{\tout}(x) := \lambda_i^{2s} \frac{\msr^{-2s}}{|y_i|^{3s}} \mone_{\{|y_i| \ge \msr^2\}}(x)
\end{cases}
\end{equation}
for $n = 6s$ and $i = 1,\ldots,\nu$.
\end{defn}
\noindent Clearly, the norms $\|\cdot\|_*$ and $\|\cdot\|_{**}$ depend on the choice of $z_i \in \R^n$ and $\lambda_i \in (0,\infty)$.
If we keep using $v_i^{\tout}$ and $w_i^{\tout}$ in \eqref{eq:vwinout} for $n = 6s$, their slow decay, specifically $|y_i|^{-4s}$ and $|y_i|^{-2s}$, causes additional technical complexity that does not arise when using \eqref{eq:vwinout2}.

By utilizing the above norms, we will derive \eqref{eq:main} for all $n \ge 6s$ and small $\delta > 0$. The derivation is split into three steps.

\medskip \noindent \doublebox{\textsc{Step 1.}} Let $\sigma = \sum_{i=1}^{\nu} U[z_i,\lambda_i] = \sum_{i=1}^{\nu} U_i$ be such that
\[\|u - \sigma\|_{\dot{H}^s(\R^n)} = \inf_{(\tz_1, \ldots, \tz_{\nu},\, \tla_1, \ldots, \tla_{\nu})
\in \R^{n\nu} \times (0,\infty)^{\nu}} \left\|u - \sum_{i=1}^{\nu} U[\tz_i,\tla_i]\right\|_{\dot{H}^s(\R^n)} \le \delta.\]
We also set $\rho = u-\sigma \in \dot{H}^s(\R^n)$ and $Z_i^a = Z^a[z_i,\lambda_i]$ for $a = 1,\ldots,n+1$.
Because of \eqref{eq:mainc}, the family $\{U_i\}_{i = 1,\ldots,\nu}$ is $\delta'$-interacting for some $\delta' > 0$ where $\delta' \to 0$ as $\delta \to 0$. The function $\rho$ satisfies
\begin{equation}\label{eq:30}
(-\Delta)^s \rho - \left[|\sigma+\rho|^{p-1}(\sigma+\rho) - \sigma^p\right] = \(\sigma^p-\sum_{i=1}^{\nu} U_i^p\) + \left[(-\Delta)^su - |u|^{p-1}u\right]
\end{equation}
in $\R^n$ and
\begin{equation}\label{eq:301}
\int_{\R^n} (-\Delta)^{s \over 2} \rho\, (-\Delta)^{s \over 2} Z_i^a \dx = \int_{\R^n} \rho U_i^{p-1} Z_i^a \dx = 0 \quad \text{for all } i = 1,\ldots,\nu \text{ and } a = 1,\ldots,n+1.
\end{equation}

Consider the equation
\begin{equation}\label{eq:31}
\begin{cases}
\displaystyle (-\Delta)^s \rho_0 - \left[|\sigma+\rho_0|^{p-1}(\sigma+\rho_0) - \sigma^p\right]
= \(\sigma^p-\sum_{i=1}^{\nu} U_i^p\) + \sum_{i=1}^{\nu} \sum_{a=1}^{n+1} c_i^a U_i^{p-1} Z_i^a \quad \text{in } \R^n,\\
\displaystyle \rho_0 \in \dot{H}^s(\R^n),\, c_1^1, \ldots, c_{\nu}^{n+1} \in \R,\\
\displaystyle \int_{\R^n} \rho_0 U_i^{p-1} Z_i^a \dx = 0 \quad \text{for all } i = 1,\ldots,\nu \text{ and } a = 1,\ldots,n+1.
\end{cases}
\end{equation}
To solve \eqref{eq:31}, we will use a pointwise estimate on the error term $\sigma^p-\sum_{i=1}^{\nu} U_i^p$.
\begin{lemma}\label{lemma:32}
There exists a constant $C > 0$ depending only on $n$, $s$, and $\nu$ such that
\begin{equation}\label{eq:310}
\left\|\sigma^p-\sum_{i=1}^{\nu} U_i^p\right\|_{**} \le C
\end{equation}
provided $\delta > 0$ small.
\end{lemma}
\begin{proof}
The proof is essentially the same as that of \cite[Proposition 3.4]{DSW}, which we omit.
\end{proof}
\noindent In addition, we analyze an associated inhomogeneous equation
\begin{equation}\label{eq:lin}
\begin{cases}
\displaystyle (-\Delta)^s f - p\,\sigma^{p-1}f = h + \sum_{i=1}^{\nu} \sum_{a=1}^{n+1} c_i^a U_i^{p-1} Z_i^a \quad \text{in } \R^n,\\
\displaystyle f \in \dot{H}^s(\R^n),\, c_1^1, \ldots, c_{\nu}^{n+1} \in \R,\\
\displaystyle \int_{\R^n} f U_i^{p-1} Z_i^a \dx = 0 \quad \text{for all } i = 1,\ldots,\nu \text{ and } a = 1,\ldots,n+1.
\end{cases}
\end{equation}

\begin{prop}\label{prop:31}
If $\|h\|_{**} < \infty$ and $f$ satisfies \eqref{eq:lin}, then $\|f\|_* < \infty$.
Moreover, there exists a constant $C > 0$ depending only on $n$, $s$, and $\nu$ such that
\begin{equation}\label{eq:311}
\|f\|_* \le C\|h\|_{**}
\end{equation}
provided $\delta > 0$ small.
\end{prop}
\noindent Deducing the above proposition is the most challenging part of the entire proof.
Because of its complexity and length, we will put it off until Section \ref{sec:dimhigh2}.

\medskip
From Lemma \ref{lemma:32} and Proposition \ref{prop:31}, we establish the unique existence of a solution to \eqref{eq:31}.
\begin{prop}\label{prop:33}
Assume that $\delta > 0$ is small enough. Equation \eqref{eq:31} has a solution $\rho_0$ and a family $\{c_i^a\}_{i=1,\ldots,\nu,\ a=1,\ldots,n+1}$ of numbers such that
\begin{equation}\label{eq:33}
\|\rho_0\|_* \le C
\end{equation}
where $C > 0$ depends only on $n$, $s$, and $\nu$. Besides,
\begin{equation}\label{eq:331}
\|\rho_0\|_{\dot{H}^s(\R^n)} \le C\begin{cases}
\msq^{p \over 2} &\text{for } n > 6s,\\
\msq|\log \msq|^{\frac{1}{2}} &\text{for } n = 6s
\end{cases}
\end{equation}
where $\msq > 0$ is the value in \eqref{eq:Q}.
\end{prop}
\begin{proof}
A priori estimate \eqref{eq:311} and the Fredholm alternative imply that a solution $f$ to \eqref{eq:lin} uniquely exists for a given $h$ with $\|h\|_{**} < \infty$.
Therefore, relying on Lemma \ref{lemma:32} and the fact that the main order term of $|\sigma+\rho_0|^{p-1}(\sigma+\rho_0) - \sigma^p$ is $p\, \sigma^{p-1} \rho_0$,
we can apply a fixed point argument to yield the unique existence of $\rho_0$ and $\{c_i^a\}$ satisfying \eqref{eq:31} and \eqref{eq:33}.
By testing \eqref{eq:30} with $\rho_0$ and employing \eqref{eq:33}, we also discover \eqref{eq:331}.
For details, refer to the proof of Lemma 5.2 and Propositions 5.3, 5.4, and 6.1 in \cite{DSW} in which the case $s = 1$ is treated.
\end{proof}

\medskip \noindent \doublebox{\textsc{Step 2.}} Set $\rho_1 = \rho-\rho_0$. In light of \eqref{eq:30}, \eqref{eq:301}, and \eqref{eq:31}, we have
\begin{equation}\label{eq:34}
\begin{cases}
\begin{aligned}
(-\Delta)^s \rho_1 - \left[|\sigma+\rho_0+\rho_1|^{p-1}(\sigma+\rho_0+\rho_1) - |\sigma+\rho_0|^{p-1}(\sigma+\rho_0)\right] \\
= \left[(-\Delta)^su - |u|^{p-1}u\right] - \sum_{i=1}^{\nu} \sum_{a=1}^{n+1} c_i^a U_i^{p-1} Z_i^a
\end{aligned} \quad \text{in } \R^n,\\
\displaystyle \rho_1 \in \dot{H}^s(\R^n),\, c_1^1, \ldots, c_{\nu}^{n+1} \in \R,\\
\displaystyle \int_{\R^n} \rho_1 U_i^{p-1} Z_i^a \dx = 0 \quad \text{for all } i = 1,\ldots,\nu \text{ and } a = 1,\ldots,n+1.
\end{cases}
\end{equation}
\begin{prop}\label{prop:34}
Assume that $\delta > 0$ is small enough. There exists a constant $C > 0$ depending only on $n$, $s$, and $\nu$ that
\begin{equation}\label{eq:341}
\|\rho_1\|_{\dot{H}^s(\R^n)} \le C\(\Gamma(u) + \msq^2\)
\end{equation}
where $\Gamma(u) = \|(-\Delta)^su-|u|^{p-1}u\|_{\dot{H}^{-s}(\R^n)}$.
\end{prop}
\begin{proof}
Applying the spectral inequality \eqref{eq:342}, one can adapt the argument in the proof of Lemmas 6.2, 6.3, and Proposition 6.4 in \cite{DSW}. The details are omitted.
\end{proof}
Putting \eqref{eq:331} and \eqref{eq:341} together leads
\begin{equation}\label{eq:35}
\|\rho\|_{\dot{H}^s(\R^n)} \le C \begin{cases}
\Gamma(u) + \msq^{p \over 2} &\text{for } n > 6s,\\
\Gamma(u) + \msq|\log \msq|^{\frac{1}{2}} &\text{for } n=6s.
\end{cases}
\end{equation}

\medskip \noindent \doublebox{\textsc{Step 3.}} Thanks to \eqref{eq:35}, we only need to check that $\msq \lesssim \Gamma(u)$ to establish \eqref{eq:main}.

We have
\begin{equation}\label{eq:36}
\left||\sigma+\rho|^{p-1}(\sigma+\rho) - \sigma^p - p\, \sigma^{p-1}\rho\right| \lesssim \sigma^{p-2}\rho^2 \quad \text{for any } p \in (1,2]\ (\Leftrightarrow n \ge 6s).
\end{equation}
For any $j = 1,\ldots,\nu$ and $a = 1,\ldots,n+1$,
\begin{equation}\label{eq:362}
\(\sigma^{p-1} - U_j^{p-1}\) \left|Z_j^a\right| \lesssim \(\sigma^{p-1} - U_j^{p-1}\) U_j \lesssim \sum_{i=1}^{\nu} \(\sigma^{p-1} - U_i^{p-1}\) U_i = \sigma^p - \sum_{i=1}^{\nu} U_i^p,
\end{equation}
so
\begin{equation}\label{eq:363}
\begin{aligned}
\int_{\R^n} \(\sigma^{p-1} - U_j^{p-1}\) |\rho_0| \left|Z_j^{n+1}\right| \dx &\lesssim \int_{\R^n} \(\sigma^p - \sum_{i=1}^{\nu} U_i^p\) |\rho_0| \dx \\
&\lesssim \int_{\R^n} \mcv \mcw \dx \le \|\mcv\|_{L^{\frac{p+1}{p}}(\R^n)}\|\mcw\|_{L^{p+1}(\R^n)} \\
&\lesssim \begin{cases}
\msr^{-\frac{n+2s}{2}} \cdot \msr^{-\frac{n+2s}{2}} \simeq \msq^p &\text{for } n > 6s, \\
\msr^{-8s} \log \msr \simeq \msq^2|\log \msq| &\text{for } n=6s.
\end{cases}
\end{aligned}
\end{equation}
Here, the second inequality in \eqref{eq:363} is a consequence of \eqref{eq:310} and \eqref{eq:33},
and the fourth inequality can be achieved through straightforward computations; refer to \cite[Lemma 3.7]{DSW} for $s = 1$. We also used \eqref{eq:R} in the last line.

By testing \eqref{eq:30} with $Z_j^{n+1}$ for any fixed $j = 1,\ldots,\nu$, and applying \eqref{eq:36}, \eqref{eq:363},
H\"older's inequality, \eqref{eq:Sobolev}, \eqref{eq:341}, and \eqref{eq:35}, we observe
\begin{align}
&\ \left|\int_{\R^n} \(\sigma^p-\sum_{i=1}^{\nu} U_i^p\)Z_j^{n+1} \dx\right| \nonumber \\
&\lesssim \int_{\R^n} \(\sigma^{p-1} - U_j^{p-1}\) |\rho_0| \left|Z_j^{n+1}\right| \dx
+ \int_{\R^n} \sigma^{p-1} |\rho_1| \left|Z_j^{n+1}\right| \dx + \int_{\R^n} \sigma^{p-2}\rho^2 \left|Z_j^{n+1}\right| \dx + \Gamma(u) \nonumber \\
&\lesssim \int_{\R^n} \(\sigma^p-\sum_{u=1}^{\nu} U_i^p\) |\rho_0| \dx + \int_{\R^n} \sigma^p |\rho_1| \dx + \int_{\R^n} \sigma^{p-1}\rho^2 \dx + \Gamma(u) \label{eq:364} \\
&\lesssim \left\{\!\begin{aligned}
&\msq^p &&\text{for } n > 6s,\\[1ex]
&\msq^2|\log \msq| &&\text{for } n = 6s
\end{aligned}\right\} + \|\rho_1\|_{\dot{H}^s(\R^n)} + \|\rho\|_{\dot{H}^s(\R^n)}^2 + \Gamma(u) \nonumber \\
&\lesssim \Gamma(u) + \begin{cases}
\msq^p &\text{for } n > 6s,\\
\msq^2|\log \msq| &\text{for } n = 6s.
\end{cases} \nonumber
\end{align}
Furthermore, the proof of \cite[Lemma 2.1]{DSW} shows
\begin{equation}\label{eq:365}
\int_{\R^n} \(\sigma^p-\sum_{i=1}^{\nu} U_i^p\)Z_j^{n+1} \dx = \sum_{\substack{i = 1,\ldots,\nu,\\i \ne j}} \int_{\R^n} U_i^p \lambda_j \pa_{\lambda_j} U_j \dx + o(\msq) \quad \text{for all } j = 1,\ldots,\nu,
\end{equation}
which together with \eqref{eq:364} yields
\begin{equation}\label{eq:366}
\Bigg|\sum_{\substack{i = 1,\ldots,\nu,\\i \ne j}} \int_{\R^n} U_i^p \lambda_j \pa_{\lambda_j} U_j \dx\Bigg| \lesssim \Gamma(u) + o(\msq) \quad \text{for all } j = 1,\ldots,\nu
\end{equation}
where $o(\msq)$ is a term such that $o(\msq)/\msq \to 0$ as $\msq \to 0$. As can be seen in the proof of \cite[Lemma 2.3]{DSW}, one can draw the desired inequality $\msq \lesssim \Gamma(u)$ from \eqref{eq:366}.
This completes the proof of \eqref{eq:main} for $n \ge 6s$ under the validity of Proposition \ref{prop:31}.

The sharpness of \eqref{eq:main} can be proven as in \cite[Section 7]{DSW}, which we omit.

\section{Quantitative stability estimate for dimension $n \ge 6s$ (2)}\label{sec:dimhigh2}
This section is devoted to the proof of Proposition \ref{prop:31} for $n \ge 6s$. We divide it into three substeps.

\medskip \noindent \fbox{\textsc{Substep 1.}} We verify the first claim in the statement of Proposition \ref{prop:31}.
\begin{lemma}\label{lemma:3c1}
If $\|h\|_{**} < \infty$ and $f$ satisfies \eqref{eq:lin}, then $\|f\|_* < \infty$.
\end{lemma}
\begin{proof}
Suppose first that $n > 6s$. It suffices to check that
\begin{equation}\label{eq:312}
f \in L^{2n \over n-2s}(\R^n) \text{ and } \left\|\la\, \cdot\, \ra^{n-2s} h\right\|_{L^{\infty}(\R^n)} < \infty \Rightarrow \left\|\la\, \cdot\, \ra^{n-4s} f\right\|_{L^{\infty}(\R^n)} < \infty.
\end{equation}

Following the proof of \cite[Theorem 4.5]{CLO} and exploiting $n > 4s$ to control the term $h$, we get the integral representation of $f$ from \eqref{eq:lin}:
\begin{equation}\label{eq:320}
f = \Phi_{n,s} \ast \(p\,\sigma^{p-1}f + h + \sum_{i=1}^{\nu} \sum_{a=1}^{n+1} c_i^a U_i^{p-1} Z_i^a\) \quad \text{in } \R^n
\end{equation}
where $\Phi_{n,s}$ is the Riesz potential in \eqref{eq:Riesz}. By virtue of the hypothesis on $h$ in \eqref{eq:312},
there exists a large constant $c > 0$ depending only on $n$, $s$, $\nu$, $h$, $z_i$, $\lambda_i$, and $c_i^a$ such that
\[|f(x)| \le \int_{\R^n} \frac{c}{|x-\om|^{n-2s}} \left[\frac{|f(\om)|}{\la \om \ra^{4s}} + \frac{1}{\la \om \ra^{n-2s}}\right] \dom \quad \text{for } x \in \R^n.\]
Since
\begin{equation}\label{eq:315}
\int_{\R^n} \frac{1}{|x-\om|^{n-2s}} \frac{\dom}{\la \om \ra^{n-2s}} \lesssim \frac{1}{|x|^{n-4s}}
\quad \text{for all } x \in \R^n \text{ with } d := \frac{|x|}{2} \ge 1,
\end{equation}
we have
\begin{equation}\label{eq:313}
|f(x)| \le c \left[\int_{\R^n} \frac{1}{|x-\om|^{n-2s}} \frac{|f(\om)|}{\la \om \ra^{4s}} \dom + \frac{1}{\la x \ra^{n-4s}}\right]
\quad \text{for } x \in \R^n.
\end{equation}
Hence, by the HLS inequality,
\begin{equation}\label{eq:314}
\|f\|_{L^{t^*}(\R^n)} \lesssim c\(\left\||f| \ast \frac{1}{|\cdot|^{n-2s}}\right\|_{L^{t^*}(\R^n)}
+ \left\|\frac{1}{\la\, \cdot\, \ra^{n-4s}}\right\|_{L^{t^*}(\R^n)}\) \lesssim c \(\|f\|_{L^t(\R^n)} + 1\)
\end{equation}
for any $t \in [\frac{2n}{n-2s}, \frac{n}{2s})$ (which is a non-empty interval for $n > 6s$) and $t^* = \frac{nt}{n-2st}$.
By employing \eqref{eq:314} and arguing as follows, we can show that $f \in L^{\tit}(\R^n)$ for all $\tit \ge \frac{2n}{n-2s}$:
\begin{itemize}
\item[-] We take $t = t_1 := \frac{2n}{n-2s}$ in \eqref{eq:314}.
    Then $f \in L^{t_2}(\R^n)$ for $t^* = t_2 := \frac{nt_1}{n-2st_1}$, and so $f \in L^{\tit}(\R^n)$ for all $\tit \in [t_1,t_2]$. We check whether $t_2 \ge \frac{n}{2s}$ or not.
\item[-] If $t_2 \ge \frac{n}{2s}$, we put $t = \frac{n}{2s} - \ep$ for any small $\ep > 0$ into \eqref{eq:314}. It implies that $f \in L^{\tit}(\R^n)$ for all $\tit \ge \frac{2n}{n-2s}$.
\item[-] If $t_2 < \frac{n}{2s}$, we plug $t = t_2$ into \eqref{eq:314}. It gives that $f \in L^{\tit}(\R^n)$
    for all $\tit \in [t_2,t_3]$ where $t_3 := \frac{nt_2}{n-2st_2}$. We check whether $t_3 \ge \frac{n}{2s}$ or not.
\item[-] We iterate the above process. It terminates in a finite step because $t_{m+1} \ge (1+\frac{4s}{n-6s})t_m$ for all $m$.
\end{itemize}
Let us fix some $t \gg 1$ large enough. Computing as in \eqref{eq:315} and writing the H\"older conjugate of $t$ as $t'$, we find
\begin{equation}\label{eq:318}
\begin{aligned}
\int_{\R^n} \frac{1}{|x-\om|^{n-2s}} \frac{|f(\om)|}{\la \om \ra^{4s}} \dom &\lesssim \(\int_{\R^n} \frac{1}{|x-\om|^{(n-2s)t'}} \frac{\dom}{\la \om \ra^{4st'}}\)^{1 \over t'} \|f\|_{L^t(\R^n)} \\
&\lesssim \frac{1}{\la x \ra^{\frac{n(t'-1)}{t'}+2s}} \lesssim 1 \quad \text{for } x \in \R^n.
\end{aligned}
\end{equation}
From this and \eqref{eq:313}, we deduce that $f \in L^{\infty}(\R^n)$.
Putting this fact into \eqref{eq:313} and working as in \eqref{eq:315} produce
\begin{equation}\label{eq:319}
|f(x)| \lesssim c \left[\int_{\R^n} \frac{1}{|x-\om|^{n-2s}} \frac{\dom}{\la \om \ra^{4s}} + \frac{1}{\la x \ra^{n-4s}}\right] \lesssim \frac{c}{\la x \ra^{\min\{2s,n-4s\}}} \quad \text{for } x \in \R^n.
\end{equation}
Feeding this back to \eqref{eq:313}, we further obtain that $\|\la\, \cdot\, \ra^{\min\{4s,n-4s\}} f\|_{L^{\infty}(\R^n)} < \infty$.
Repeating this process finitely many times, we conclude that the estimate for $f$ in \eqref{eq:312} is true.

\medskip
If $n = 6s$, we still have \eqref{eq:313} with $\la x \ra^{n-4s}$ replaced by $\la x \ra^{3s}$. However, we cannot proceed as in \eqref{eq:314}, because $[\frac{2n}{n-2s},\frac{n}{2s}) = [3,3) = \emptyset$.
Fortunately, thanks to $f \in \dot{H}^s(\R^n) \subset L^3(\R^n)$, the HLS inequality, and H\"older's inequality, we see
\begin{equation}\label{eq:414}
\begin{aligned}
\|f\|_{L^{t^*}(\R^n)} &\lesssim c\(\left\|\frac{|f|}{\la\, \cdot\,\ra^{4s}} \ast \frac{1}{|\cdot|^{4s}}\right\|_{L^{t^*}(\R^n)} + \left\|\frac{1}{\la\, \cdot\, \ra^{3s}}\right\|_{L^{t^*}(\R^n)}\) \\
&\lesssim c \(\left\|\frac{|f|}{\la\, \cdot\, \ra^{4s}}\right\|_{L^{\zeta_2}(\R^n)} + 1\)
\lesssim c \(\|f\|_{L^3(\R^n)} \left\|\frac{1}{\la\, \cdot\, \ra^{4s}}\right\|_{L^{\zeta_1}(\R^n)} + 1\)
\end{aligned}
\end{equation}
for $\zeta_1 \in (3,\infty)$, $\zeta_2 = \frac{3\zeta_1}{\zeta_1+3} \in (\frac{3}{2},3)$, and $t^* = \frac{3\zeta_2}{3-\zeta_2} \in (3,\infty)$.
This means that $f \in L^{\tit}(\R^n)$ for all $\tit \ge 3$.\footnote{Unlike \eqref{eq:314}, we cannot ignore the factor $\la \om \ra^{-4s}$ to get meaningful information through \eqref{eq:414}.
On the other hand, without appealing to the iteration process as in Substep 1 of the proof of Proposition \ref{prop:31},
we can directly achieve $f \in L^{t^*}(\R^n)$ for any large $t^*$ by choosing $\zeta_1 > 0$ large.}
As in the case $n > 6s$, we conclude that $\|\la\, \cdot\, \ra^{3s} f\|_{L^{\infty}(\R^n)} < \infty$.
\end{proof}

\medskip \noindent \fbox{\textsc{Substep 2.}} We estimate the coefficients $c_i^a$'s in \eqref{eq:lin}.
\begin{lemma}\label{lemma:3c2}
There is a constant $C > 0$ depending only on $n$, $s$, and $\nu$ such that
\begin{equation}\label{eq:cjb}
\sum_{i=1}^{\nu} \sum_{a=1}^{n+1} |c_i^a| \le C\(\|h\|_{**} \msr^{2s-n} + \|f\|_* \times \left\{\!\begin{aligned}
&\msr^{-(n+2s)} &&\text{for } n > 6s,\\[1ex]
&\msr^{-8s} \log \msr &&\text{for } n = 6s
\end{aligned}\right\}\)
\end{equation}
provided $\delta > 0$ small.
\end{lemma}
\begin{proof}
The proof of \eqref{eq:cjb} is essentially the same as that of \cite[Lemma 5.2]{DSW}, so we skip it.
\end{proof}

\medskip \noindent \fbox{\textsc{Substep 3.}} We prove that \eqref{eq:311} holds for $\delta > 0$ sufficiently small.

\medskip
Suppose that \eqref{eq:311} is false. By virtue of Lemma \ref{lemma:3c1}, there exist sequences of small positive numbers $\{\delta'_m\}_{m \in \N}$,
$\delta'_m$-interacting families $\{\{U_{i,m} = U[z_{i,m},\lambda_{i,m}]\}_{i = 1,\ldots,\nu}\}_{m \in \N}$, functions $\{f_m\}_{m \in \N} \subset \dot{H}^s(\R^n)$ and $\{h_m\}_{m \in \N}$,
and numbers $\{c_{i,m}^a\}_{i = 1,\ldots,\nu,\, a = 1,\ldots,n+1,\, m \in \N}$ such that
\begin{equation}\label{eq:316}
\delta'_m \to 0 \quad \text{and} \quad \|h_m\|_{**} \to 0 \quad \text{as } m \to \infty, \quad \|f_m\|_* = 1 \quad \text{for all } m \in \N,
\end{equation}
and
\begin{equation}\label{eq:317}
\begin{cases}
\displaystyle (-\Delta)^s f_m - p\, \sigma_m^{p-1}f_m = h_m + \sum_{i=1}^{\nu} \sum_{a=1}^{n+1} c_{i,m}^a U_{i,m}^{p-1} Z_{i,m}^a \quad \text{in } \R^n,\\
\displaystyle \int_{\R^n} U_{i,m}^{p-1} Z_{i,m}^a f_m \dx = 0 \quad \text{for all } i = 1,\ldots,\nu \text{ and } a = 1,\ldots,n+1.
\end{cases}
\end{equation}
Here, $\sigma_m = \sum_{i=1}^{\nu} U_{i,m}$ and $Z_{i,m}^a = Z^a[z_{i,m},\lambda_{i,m}]$.
By reordering the indices $i, j = 1,\ldots,\nu$ and taking a subsequence, one can assume that
\begin{equation}\label{eq:str}
\begin{cases}
\lambda_{1,m} \le \lambda_{2,m} \le \cdots \le \lambda_{\nu,m} \text{ for all } m \in \N,\\
\displaystyle \text{either } \lim_{m \to \infty} z_{ij,m} = z_{ij,\infty} \in \R^n \text{ or } \lim_{m \to \infty} |z_{ij,m}| \to \infty
\end{cases}
\end{equation}
where $z_{ij,m} := \lambda_{i,m}(z_{j,m}-z_{i,m}) \in \R^n$.

Let $\mcv_m = \sum_{i=1}^{\nu}(v_{i,m}^{\tin} + v_{i,m}^{\tout})$ and $\mcw_m = \sum_{i=1}^{\nu}(w_{i,m}^{\tin} + w_{i,m}^{\tout})$ be the functions $\mcv$ and $\mcw$ in \eqref{eq:mcvw} with $(z_i,\lambda_i) = (z_{i,m},\lambda_{i,m})$, respectively. To reach a contradiction, we will establish that
\begin{equation}\label{eq:claim0}
\(|f_m|\mcw_m^{-1}\)(x) \le \tfrac{1}{2} \quad \text{for all } x \in \R^n
\end{equation}
provided $m \in \N$ large. Clearly, \eqref{eq:claim0} implies $\|f_m\|_* \le \tfrac{1}{2}$, which is absurd.

\medskip \noindent \underline{\textsc{Tree structure:}}
To prove \eqref{eq:claim0}, we will exploit the tree structure of $\delta$-interacting $\nu$-tuples of bubbles as $\delta \to 0$, as described in Lemma \ref{lemma:str} below.
The bubble-tree structure for $s = 1$ was investigated in e.g. \cite{Dr, Pr, DSW}.
The concept of bubble-trees dates back to the work of Parker and Wolfson \cite{PW} for pseudo-holomorphic maps,
and those of Parker \cite{Pa} and Qing and Tian \cite{QT} for harmonic maps on Riemann surfaces.
\begin{defn}
Let $\preceq$ be a partial order on a set $\mct$, and $\prec$ the corresponding strict partial order on $\mct$.

\medskip \noindent - A partially ordered set $(\mct,\preceq)$ is called a directed tree if for each $t \in \mct$, the set $\{s \in \mct: s \preceq t\}$ is well-ordered by the relation $\preceq$.

\medskip \noindent - A root is the least element of the set $\{s \in \mct: s \preceq t\}$ for some $t \in \mct$.

\medskip \noindent - A rooted tree is a directed tree with roots; a rooted forest is a disjoint union of rooted trees.

\medskip \noindent - A descendant of $s \in \mct$ is any element $t \in \mct$ such that $s \prec t$.
Let $\mcd(s)$ be the set of descendants of $s \in \mct$, that is, $\mcd(s) = \{t \in \mct: s \prec t\}$.
\end{defn}

\begin{lemma}\label{lemma:str}
Recalling \eqref{eq:str}, we set a relation $\prec$ on $I$ by
\[i \prec j \ \Leftrightarrow \ \left[i < j \text{ and } \lim_{m \to \infty} z_{ij,m} \in \R^n\right].\]
Then $\prec$ is a strict partial order (that corresponds to a non-strict order $\preceq$)
and there is a number $\nu^* \in \{1,\ldots,\nu\}$ such that $I$ can be expressed as a rooted forest.
\end{lemma}
\begin{proof}
A slight modification of the argument in \cite[Subsection 4.2]{DSW} works for any $s \in (0,\frac{n}{2})$.
\end{proof}

\medskip \noindent \underline{\textsc{Decomposition of $\R^n$:}}
We set $y_{i,m} = \lambda_{i,m}(x-z_{i,m}) \in \R^n$ and recall $z_{ij,m} = \lambda_{i,m}(z_{j,m}-z_{i,m}) \in \R^n$. Given any $L > 1$ large and $\vep \in (0,1)$ small, we define
\[\Omega_{i,m;L,\vep} = \left\{x \in \R^n: |y_{i,m}| \le L,\, |y_{i,m}-z_{ij,m}| \ge \vep \text{ for all } j \in \mcd(i)\right\}\]
and
\[\mca_{i,m;L,\vep} = \bigcup_{j \in \mcd(i)} \bigg[\left\{x \in \R^n: |y_{i,m}-z_{ij,m}| < \vep\right\} \setminus \bigcup_{k \in \mcd(i)} \left\{x \in \R^n: |y_{k,m}| \le L\right\}\bigg].\]
Then $\R^n$ is decomposed into three disjoint subsets:
\[\R^n = \Omega_{\ext,m;L} \cup \Omega_{\core,m;L,\vep} \cup \Omega_{\neck,m;L,\vep}\]
where
\[\left\{\begin{aligned}
&\text{Exterior region:} &&\Omega_{\ext,m;L} = \bigcap_{i=1}^{\nu} \{x \in \R^n: |y_{i,m}| > L\},\\
&\text{Core region:} &&\Omega_{\core,m;L,\vep} = \bigcup_{i=1}^{\nu} \Omega_{i,m;L,\vep},\\
&\text{Neck region:} &&\Omega_{\neck,m;L,\vep} = \bigcup_{i=1}^{\nu} \mca_{i,m;L,\vep}.
\end{aligned}\right.\]

\medskip \noindent \underline{\textsc{Preliminary estimates:}}
Let $\msr_{ij,m}$ and $\msr_m$ be the quantities in \eqref{eq:Rij}--\eqref{eq:R} where the parameter $(z_{i,m},z_{j,m},\lambda_{i,m},\lambda_{j,m})$ is substituted for $(z_i, z_j, \lambda_i, \lambda_j)$ so that $\msr_m \to \infty$ as $m \to \infty$. Then \eqref{eq:cjb} gives
\begin{equation}\label{eq:321}
\sum_{i=1}^{\nu} \sum_{a=1}^{n+1} |c_{i,m}^a| = o\(\msr_m^{2s-n}\) = o(\delta_m') \to 0 \quad \text{as } m \to \infty
\end{equation}
where $o(\msr_m^{2s-n})/\msr_m^{2s-n} \to 0$ as $m \to \infty$. By \eqref{eq:320}, \eqref{eq:316}, and \eqref{eq:321}, we know
\begin{equation}\label{eq:322}
|f_m(x)| \lesssim \int_{\R^n} \frac{1}{|x-\om|^{n-2s}} \left[\(\sigma_m^{p-1}|f_m|\)(\om) + o(1)\mcv_m(\om) + o\(\msr_m^{2s-n}\) \sum_{i=1}^{\nu} U_{i,m}^p(\om)\right] \dom
\end{equation}
for $x \in \R^n$.

We readily observe that
\begin{equation}\label{eq:325}
\int_{\R^n} \frac{1}{|x-\om|^{n-2s}} U_{i,m}^p(\om) \dom = \ga_{n,s}^{-1} \(\Phi_{n,s} \ast U_{i,m}^p\)(x) = \ga_{n,s}^{-1} U_{i,m}(x) \lesssim \frac{\lambda_{i,m}^{n-2s \over 2}}{\la y_{i,m} \ra^{n-2s}}.
\end{equation}
Moreover, the following estimates are true.
\begin{lemma}
There exists a constant $C > 0$ depending only on $n$, $s$, and $\nu$ such that
\begin{equation}\label{eq:324}
\int_{\R^n} \frac{1}{|x-\om|^{n-2s}}\(v_{j,m}^{\tin}+v_{j,m}^{\tout}\)(\om) \dom \le C\(w_{j,m}^{\tin}+w_{j,m}^{\tout}\)(x)
\end{equation}
for $j = 1,\ldots,\nu$ and
\begin{align}
&\ \frac{1}{\mcw_m(x)} \int_{\R^n} \frac{1}{|x-\om|^{n-2s}} \(\sigma_m^{p-1}\mcw_m\)(\om) \dom \label{eq:323}\\
&\le C\begin{cases}
\begin{medsize}
\displaystyle M^{3n}\sum_{i=1}^{\nu} \(\frac{1}{\la y_{i,m} \ra^{2s}} \mone_{\{|y_{i,m}| < \msr_m\}} + \frac{\log |y_{i,m}|}{|y_{i,m}|^{2s}} \mone_{\{|y_{i,m}| \ge \msr_m\}}\) + M^{4s}\msr_m^{-2s} + M^{-2s}
\end{medsize}
&\begin{medsize}
\text{for } n > 6s,
\end{medsize} \\
\begin{medsize}
\displaystyle M^{3n}\sum_{i=1}^{\nu} \(\frac{\log(2+|y_{i,m}|)}{\la y_{i,m} \ra^{s}} \mone_{\{|y_{i,m}| < \msr_m^2\}} + \frac{\log |y_{i,m}|}{|y_{i,m}|^{s}} \mone_{\{|y_{i,m}| \ge \msr_m^2\}}\) + M^{4s}\msr_m^{-2s} + M^{-2s}
\end{medsize}
&\begin{medsize}
\text{for } n = 6s
\end{medsize}
\end{cases} \nonumber
\end{align}
holds for any $x \in \R^n$, $M > 1$, and $m \in \N$ large.
\end{lemma}
\begin{proof}
A minor variation of the proof of \cite[Lemma 3.6]{DSW} yields \eqref{eq:324}. Besides, if $n > 6s$, we have the following interaction estimates as in \cite[Lemma 4.1]{DSW}: If $\lambda_{i,m} \le \lambda_{j,m}$, then
\begin{align}
U_{j,m}^{p-1}w_{i,m}^{\tin} &\lesssim \msr_{ij,m}^{-2s}v_{j,m}^{\tin} + \msr_m^{-2s}v_{j,m}^{\tout} + \msr_m^{-2s}v_{i,m}^{\tin},\label{eq:upb1}\\
U_{j,m}^{p-1}w_{i,m}^{\tout} &\lesssim \msr_{ij,m}^{-2s}v_{j,m}^{\tin} + \msr_m^{-2s}v_{j,m}^{\tout} + \msr_m^{-2s}v_{i,m}^{\tout},\label{eq:upb3}\\
U_{i,m}^{p-1}w_{j,m}^{\tin} &\lesssim \msr_{ij,m}^{-2s}v_{j,m}^{\tin},\label{eq:upb4}\\
U_{i,m}^{p-1}w_{j,m}^{\tout} &\lesssim \la z_{ij,m} \ra^{-2s} \(v_{i,m}^{\tin} + v_{i,m}^{\tout} + v_{j,m}^{\tout}\), \label{eq:upb5}
\end{align}
and
\begin{align}
U_{i,m}^{p-1}w_{j,m}^{\tout} \lesssim \left[\(\tfrac{\lambda_{i,m}}{\lambda_{j,m}}\)^2+\ep^{-2}\right]^s v_{j,m}^{\tout} &\quad \text{if } |y_{i,m}-z_{ij,m}| \le \ep^{-1},\label{eq:upb2}\\
w_{j,m}^{\tout} \lesssim \la z_{ij,m} \ra^{2n-10s} \ep^{4s-n} \(w_{i,m}^{\tin}+w_{i,m}^{\tout}\) &\quad \text{if } |y_{i,m}-z_{ij,m}| \ge \ep^{-1}\label{eq:upb6}
\end{align}
for any $\ep>0$ small and $m \in \N$ large. Using these estimates, one can mimic the proof of \cite[Proposition 4.3]{DSW} to achieve \eqref{eq:323}.

If $n=6s$, we can obtain \eqref{eq:323} by deriving analogous inequalities to \eqref{eq:upb1}--\eqref{eq:upb6} as in \cite[Lemma 4.2]{DSW}. We skip the details.
\end{proof}

In the remainder of this section, we restrict ourselves to the case $n > 6s$.
The proof for the case $n = 6s$ goes through without serious modification, so we omit it for conciseness.

\medskip
Suppose that for any given $\zeta \in (0,1)$, there exists a number $m_{\zeta} \in \N$ depending on $\zeta$ such that
\begin{equation}\label{eq:claim2}
m \ge m_{\zeta} \ \Rightarrow \ \int_{\R^n} \frac{1}{|x-\om|^{n-2s}} \(\sigma_m^{p-1}|f_m|\)(\om) \dom \le \zeta\mcw_m(x) \quad \text{for all } x \in \R^n.
\end{equation}
Then, owing to \eqref{eq:322}--\eqref{eq:324}, the desired inequality \eqref{eq:claim0} will hold. We have
\begin{equation}\label{eq:claim20}
\begin{aligned}
&\ \int_{\R^n} \frac{1}{|x-\om|^{n-2s}} \(\sigma_m^{p-1}|f_m|\)(\om) \dom \\
&= \(\int_{\Omega_{\ext,m;L}} + \int_{\Omega_{\core,m;L,\vep}} + \int_{\Omega_{\neck,m;L,\vep}}\) \frac{1}{|x-\om|^{n-2s}} \(\sigma_m^{p-1}|f_m|\)(\om) \dom \\
&=: \mci_{\ext,m;L}(x) + \mci_{\core,m;L,\vep}(x) + \mci_{\neck,m;L,\vep}(x) \quad \text{for all } x \in \R^n.
\end{aligned}
\end{equation}
By choosing suitable $L$ and $\vep$, we shall estimate each terms $\mci_{\ext,m;L}$, $\mci_{\core,m;L,\vep}$, and $\mci_{\neck,m;L,\vep}$ to deduce \eqref{eq:claim2}.

\medskip \noindent \underline{\textsc{Estimate of $\mci_{\ext,m;L}$:}}
By \eqref{eq:316} and \eqref{eq:322}--\eqref{eq:323}, there is a constant $C_0 > 0$ depending only on $n$, $s$, and $\nu$ such that
\begin{equation}\label{eq:3c11}
\sup_{x \in \Omega_{\ext,m;L}} \(|f_m|\mcw_m^{-1}\)(x) \le C_0\(M^{3n} L^{-2s}\log L + M^{4s}\msr_m^{-2s} + M^{-2s}\) + o(1).
\end{equation}
In light of \eqref{eq:3c11} and \eqref{eq:323}, there exists $C_1 > 0$ depending only on $n$, $s$, and $\nu$ such that
\begin{align*}
&\ \mci_{\ext,m;L}(x) \\
&\le \left[C_0\(M^{3n} L^{-2s}\log L + M^{4s}\msr_m^{-2s} + M^{-2s}\) + o(1)\right] \int_{\R^n} \frac{1}{|x-\om|^{n-2s}} \(\sigma_m^{p-1}\mcw_m\)(\om) \dom \nonumber \\
&\le C_1\left[C_0\(M^{3n} L^{-2s}\log L + M^{4s}\msr_m^{-2s} + M^{-2s}\) + o(1)\right] \mcw_m(x)
\end{align*}
for any $x \in \R^n$ and $m \in \N$ large. We pick numbers $M_0$ and $L_0$ so large that $C_1C_0M_0^{-2s} < \frac{\zeta}{12}$ and $C_1C_0M_0^{3n} L_0^{-2s}\log L_0 < \frac{\zeta}{12}$. Then
\begin{equation}\label{eq:claim21}
\mci_{\ext,m;L_0}(x) \le \frac{\zeta}{3} \mcw_m(x)
\end{equation}
for all $x \in \R^n$ and $m \in \N$ large.

By taking larger values for $L_0$ if required, we can assume that
\begin{equation}\label{eq:Cstar}
L_0 > C^* := 1 + \max\{|z_{ij}|: i,j=1,\ldots,\nu,\, j \in \mcd(i)\},
\end{equation}
which will be frequently used later.
Hereafter, we fix $L = L_0$ and omit the subscript $L_0$ for simplicity, writing e.g. $\Omega_{\ext,m} = \Omega_{\ext,m;L_0}$ or $\mci_{\core,m;\vep} = \mci_{\core,m;L_0,\vep}$.

\medskip \noindent \underline{\textsc{Estimate of $\mci_{\core,m;\vep}$:}} Fix any $\vep \in (0,1)$. By employing the blow-up argument, we will first show that
\begin{equation}\label{eq:claim1}
\sup_{x \in \Omega_{\core,m;\vep}}\(|f_m|\mcw_m^{-1}\)(x) = o(1) \quad \text{as } m \to \infty.
\end{equation}

If \eqref{eq:claim1} is not true, we will have points $x_m \in \Omega_{\core,m;\vep}$ for $m \in \N$ and a number $\theta_0 \in (0,1)$ such that $\theta_0 \le (|f_m|\mcw_m^{-1})(x_m) \le 1$ for all $m \in \N$.
By passing to a subsequence, we may assume that $x_m \in \Omega_{i_0,m;\vep}$ for some $i_0 = 1,\ldots,\nu$ and all $m \in \N$. The following lemma is crucial.
\begin{lemma}\label{lemma:3c6}
Let $\hf_m$ be a function in $\dot{H}^s(\R^n)$ defined as
\[\hf_m(y) = \mcw_m^{-1}(x_m) f_m\big(\lambda_{i_0,m}^{-1}y + z_{i_0,m}\big) \quad \text{for } y \in \R^n\]
and $\wtmcz_{\infty} = \{z_{i_0j,\,\infty}: j \in \mcd(i_0)\}$. Then, up to a subsequence,
\begin{equation}\label{eq:3c3}
\hf_m \to 0 \quad \text{in } C^0_{\textnormal{loc}}\(\R^n \setminus \wtmcz_{\infty}\) \quad \text{as } m \to \infty.
\end{equation}
\end{lemma}
\begin{proof}
There are several lengthy technical calculations in this proof. To make the main strategy of this proof clearer, we will postpone their derivations to Appendix \ref{sec:app2}.

\medskip We set
\begin{equation}\label{mchm}
\mch_m(y) = \lambda_{i_0,m}^{-2s} \mcw_m^{-1}(x_m) \Bigg[h_m + \sum_{j=1}^{\nu} \sum_{a=1}^{n+1} c_{j,m}^a U_{j,m}^{p-1} Z_{j,m}^a\Bigg]\big(\lambda_{i_0,m}^{-1}y+z_{i_0,m}\big) \quad \text{for } y \in \R^n.
\end{equation}
From \eqref{eq:317}, we see that
\begin{equation}\label{eq:3c30}
\begin{cases}
\displaystyle \hf_m = \Phi_{n,s} \ast \Bigg[p \left\{\lambda_{i_0,m}^{-{n-2s \over 2}} \sigma_m\big(\lambda_{i_0,m}^{-1}\cdot+z_{i_0,m}\big)\right\}^{p-1} \hf_m + \mch_m\Bigg] \quad \text{in } \R^n, \\
\displaystyle \int_{\R^n} U[0,1]^{p-1} Z^a[0,1] \hf_m \dy = 0 \quad \text{for all } i = 1,\ldots,\nu \text{ and } a = 1,\ldots,n+1.
\end{cases}
\end{equation}

Given any $M > 0$, the proof of \cite[Lemma 4.7]{DSW} shows the existence of a sequence $\{\eta_{M,m}\}_{m \in \N} \subset (0,\infty)$ such that $\eta_{M,m} \to 0$ as $m \to \infty$ and
\begin{equation}\label{eq:usua}
\(w_{j,m}^{\tin}+w_{j,m}^{\tout}\)\big(\lambda_{i_0,m}^{-1}y+z_{i_0,m}\big) = \eta_{M,m} w_{i_0,m}^{\tin}\big(\lambda_{i_0,m}^{-1}y+z_{i_0,m}\big)
\end{equation}
for all $y \in B(0,M)$, $j \notin \mcd(i_0)$, and $m \in \N$ large.\footnote{In
principle, we may have that $\liminf_{m \to \infty}\sup_{M > 0} \eta_{M,m} = \infty$ because of the presence of bubbles that belong to a different bubble-tree than the one associated with the index $i_0$.}
Using the elementary inequality
\[\frac{\sum_{i=1}^{\nu} a_i}{\sum_{i=1}^{\nu} b_i} \le \max\left\{\frac{a_1}{b_1}, \ldots, \frac{a_{\nu}}{b_{\nu}}\right\} \le \sum_{i=1}^{\nu} \frac{a_i}{b_i} \quad \text{for } a_1, \ldots, a_{\nu} \ge 0 \text{ and } b_1, \ldots, b_{\nu} > 0\]
and \eqref{eq:Cstar}, we verify that
\begin{align}
&\begin{medsize}
\displaystyle \ \frac{\mcw_m\big(\lambda_{i_0,m}^{-1}y+z_{i_0,m}\big)}{\mcw_m(x_m)} \le \frac{\left[(1+\nu\eta_{M,m}) w_{i_0,m}^{\tin} + \sum_{j \in \mcd(i_0)}\(w_{j,m}^{\tin}+w_{j,m}^{\tout}\)\right]
\big(\lambda_{i_0,m}^{-1}y+z_{i_0,m}\big)}{\left[w_{i_0,m}^{\tin} + \sum_{j \in \mcd(i_0)} \(w_{j,m}^{\tin}+w_{j,m}^{\tout}\)\right](x_m)}
\end{medsize} \label{eq:3c32} \\
&\begin{medsize}
\displaystyle \lesssim (1+\nu\eta_{M,m}) L_0^{2s} + \sum_{j \in \mcd(i_0)} \left[\frac{L_0^{2s}}{|y-z_{i_0j,m}|^{2s}} \mone_{\left\{\left|\frac{\lambda_{j,m}}{\lambda_{i_0,m}}(y-z_{i_0j,m})\right| < \msr_m\right\}}
+ \frac{L_0^{n-4s}}{|y-z_{i_0j,m}|^{n-4s}} \mone_{\left\{\left|\frac{\lambda_{j,m}}{\lambda_{i_0,m}}(y-z_{i_0j,m})\right| \ge \msr_m\right\}}\right]
\end{medsize} \nonumber
\end{align}
for $y \in B(0,M) \setminus \{z_{i_0j,m}: j \in \mcd(i_0)\}$. Given any $l > 0$, let
\[\mck_l := \left\{y \in \R^n: |y|\le l,\, |y-z_{i_0j,\infty}|\ge l^{-1} \text{ for all } j \in \mcd(i_0)\right\} \subset \R^n \setminus \wtmcz_{\infty}.\]
Then there exists a large number $m_l \in \N$ depending on $l$ such that
\[y \in \mck_l,\, j \in \mcd(i_0),\, m \ge m_l \ \Rightarrow \ \left|\frac{\lambda_{j,m}}{\lambda_{i_0,m}}(y-z_{i_0j,\infty})\right| \ge \frac{1}{2} l^{-1} \sqrt{\frac{\lambda_{j,m}}{\lambda_{i_0,m}}}\msr_m \gg \msr_m.\]
Thus \eqref{eq:3c32} gives
\begin{equation}\label{eq:hfmp}
\big|\hf_m(y)\big| \lesssim (1+\nu\eta_{l,m})L_0^{2s} + \sum_{j \in \mcd(i_0)} \frac{L_0^{n-4s}}{|y-z_{i_0j,m}|^{n-4s}} \quad \text{uniformly in } \mck_l \text{ for } m \ge m_l.
\end{equation}
In Appendix \ref{subsec:app22}, we will directly use \eqref{eq:3c30} to prove that if $s \in (\frac{1}{2},\frac{n}{2})$, then
\begin{equation}\label{eq:3c35}
\sup_{m \in \N} \big\|\hf_m\big\|_{C^1\(\overline{B'}\)} \lesssim 1
\end{equation}
for any open ball $B' \subset \mck_l$. For $s \in (0,\frac{1}{2}]$, the standard elliptic regularity (see e.g. \cite[Corollary 2.5]{RS}) yields
\begin{equation}\label{eq:3c352}
\sup_{m \in \N} \big\|\hf_m\big\|_{C^{0,\alpha}\(\overline{B'}\)} \lesssim 1 \quad \text{for all } \alpha \in (0,2s).
\end{equation}
Then, by employing the diagonal argument with \eqref{eq:3c35}--\eqref{eq:3c352}, we obtain
\begin{equation}\label{eq:3c36}
\hf_m \to \hf_{\infty} \quad \text{in } C^0_{\text{loc}}\(\R^n \setminus \wtmcz_{\infty}\) \quad \text{as } m \to \infty
\end{equation}
for some function $\hf_{\infty}$. From \eqref{eq:hfmp}, we see that
\[\big|\hf_{\infty}(y)\big| \lesssim L_0^{2s} + \sum_{j \in \mcd(i_0)} \frac{L_0^{n-4s}}{|y-z_{i_0j,\,\infty}|^{n-4s}} \quad \text{in } \R^n\setminus \wtmcz_{\infty}.\]
As shown in Appendix \ref{subsec:app23} and \ref{apsu3}, the above observation will lead
\begin{equation}\label{eq:3c37}
\hf_{\infty} = \Phi_{n,s} \ast \(p U[0,1]^{p-1} \hf_{\infty}\) \quad \text{in } \R^n \setminus \wtmcz_{\infty}
\end{equation}
and
\begin{equation}\label{eq:last}
\int_{\R^n} U[0,1]^{p-1} Z^a[0,1] \hf_{\infty} \dy = 0 \quad \text{for all } i = 1,\ldots,\nu \text{ and } a = 1,\ldots,n+1.
\end{equation}
Now, Lemma \ref{lemma:reg2} implies that each singularity $z_{i_0j,\,\infty}$ of $\hf_{\infty}$ is removable, namely, $\hf_{\infty}$ extends to a function in $L^{\infty}(\R^n)$ satisfying \eqref{eq:3c37} in $\R^n$.
By the non-degeneracy of the bubble $U[0,1]$ shown in \cite[Lemma 5.1]{LX}, it follows that $\hf_{\infty} = 0$ in $\R^n$. As a result, \eqref{eq:3c3} must hold.
\end{proof}

Let $Y_m = \lambda_{i_0,m}(x_m-z_{i_0,m})$ so that $\sup_{m \in \N} |Y_m| \le L_0 < \infty$.
One can assume that $Y_m \to Y_{\infty}$ as $m \to \infty$ for some $Y_{\infty} \in \R^n$ such that $|Y_{\infty}| \le L_0$ and $|Y_{\infty}-z_{i_0j,\infty}| \ge \vep$ for all $j \in \mcd(i_0)$.
By \eqref{eq:3c3}, one concludes that $\hf_m(Y_m) \to 0$ as $m \to \infty$, which is impossible because $|\hf_m(Y_m)| = (|f_m|\mcw_m^{-1})(x_m) \ge \theta_0 > 0$. Therefore, \eqref{eq:claim1} is true.

\medskip
Now, from \eqref{eq:claim1} and \eqref{eq:323}, we infer that
\begin{equation}\label{eq:claim22}
\mci_{\core,m;\vep}(x) = o(1) \int_{\Omega_{\core,m;\vep}} \frac{1}{|x-\om|^{n-2s}} \(\sigma_m^{p-1}\mcw_m\)(\om) \dom
= o(1) \mcw_m(x)
\end{equation}
for all $x \in \R^n$ and $m \in \N$ large.

\medskip \noindent \underline{\textsc{Estimate of $\mci_{\neck,m;\vep}$:}} By reasoning as in Case 3 of the proof of \cite[Lemma 5.1]{DSW}, we find that for any fixed $\theta \in (0,1)$ and $i = 1,\ldots,\nu$,
\begin{equation}\label{eq:claim231}
\begin{aligned}
\sigma_m^{p-1}\mcw_m &\lesssim \left[(C^*)^{\frac{2s(n-6s)}{n-2s}}\theta + L_0^{-2s} + o(1)\right] \sum_{j \in \mcd(i)}\(v_{j,m}^{\tin}+v_{j,m}^{\tout}\) \\
&\ + \left[(C^*)^{\frac{2s(n-6s)}{n-2s}} \theta^{-\frac{n-4s}{2s}} + o(1)\right]v_{i,m}^{\tin} \quad \text{in } \mca_{i,m;\vep},
\end{aligned}
\end{equation}
provided $L_0 > 3C^*$ (see \eqref{eq:Cstar} for the definition of $C^*$) and $m \in \N$ large. % Clearly, $3C^* > \vep+C^*$.
By \eqref{eq:claim231},
\begin{align*}
&\ \mci_{\neck,m;\vep}(x) \\
&\le \sum_{i=1}^{\nu} \int_{\mca_{i,m;\vep}} \frac{1}{|x-\om|^{n-2s}} \(\sigma_m^{p-1}\mcw_m\)(\om) \dom\\
&\lesssim \sum_{i=1}^{\nu} \int_{\mca_{i,m;\vep}} \frac{1}{|x-\om|^{n-2s}}
\left[\left\{(C^*)^{2s}\theta + L_0^{-2s} + o(1)\right\} \mcv_m(\om) + (C^*)^{2s}\theta^{-\frac{n-4s}{2s}} v_{i,m}^{\tin}(\om)\right] \dom \\
&\lesssim (C^*)^{2s}\theta^{-\frac{n-4s}{2s}} \sum_{i=1}^{\nu} \int_{\mca_{i,m;\vep}} \frac{1}{|x-\om|^{n-2s}} v_{i,m}^{\tin}(\om) \dom \\
&\ + \left\{(C^*)^{2s}\theta + L_0^{-2s} + o(1)\right\} \int_{\R^n} \frac{1}{|x-\om|^{n-2s}} \mcv_m(\om) \dom.
\end{align*}
Hence, by applying \eqref{eq:324} and possibly increasing the value of $L_0$, we achieve
\begin{equation}\label{eq:claim233}
\mci_{\neck,m;\vep}(x) \le C(C^*)^{2s}\theta^{-\frac{n-4s}{2s}} \sum_{i=1}^{\nu} \int_{\mca_{i,m;\vep}} \frac{1}{|x-\om|^{n-2s}} v_{i,m}^{\tin}(\om) \dom + \frac{\zeta}{6} \mcw_m(x)
\end{equation}
for any $x \in \R^n$ and $\theta \in (0,1)$ small. Moreover,
\begin{equation}\label{eq:claim234}
\begin{aligned}
&\ \int_{\mca_{i,m;\vep}} \frac{1}{|x-\om|^{n-2s}} v_{i,m}^{\tin}(\om) \dom \\
&\le \sum_{j \in \mcd(i)} \underbrace{\int_{B(z_{j,m},\frac{\vep}{\lambda_{i,m}}) \setminus \cup_{k\in\mcd(i)}B(z_{k,m},\frac{L_0}{\lambda_{k,m}})} \frac{1}{|x-\om|^{n-2s}}
\frac{\lambda_{i,m}^{\frac{n+2s}{2}}\msr_m^{2s-n}}{\la\lambda_{i,m}(\om-z_{i,m})\ra^{4s}} \mone_{B(z_{i,m},\frac{\msr_m}{\lambda_{i,m}})}(\om) \dom}_{=: \mcj_{i,m;\vep}(x)}
\end{aligned}
\end{equation}
for $x \in \R^n$. We will estimate $\mcj_{i,m;\vep}(x)$ by considering three separate cases.

\medskip \noindent \textit{Case 1:} Fixing any $L' \ge 2C^*$, we assume that $|x-z_{j,m}| \ge \frac{L'}{\lambda_{i,m}}$ for some $j \in \mcd(i)$.

Letting $\ty_{ij,m} =\lambda_{i,m}(x-z_{j,m})$ and $\tom_{ij,m} = \lambda_{i,m}(\om-z_{j,m})$, and recalling $z_{ij,m} = \lambda_{i,m}(z_{j,m}-z_{i,m})$, we evaluate
\begin{equation}\label{eq:claim232}
\begin{aligned}
\mcj_{i,m;\vep}(x) &\le \int_{B(z_{j,m},\frac{\vep}{\lambda_{i,m}})} \frac{1}{|x-\om|^{n-2s}} \frac{\lambda_{i,m}^{\frac{n+2s}{2}}\msr_m^{2s-n}}{\la\lambda_{i,m}(\om-z_{i,m})\ra^{4s}} \mone_{B(z_{i,m},\frac{\msr_m}{\lambda_{i,m}})}(\om)\dom \\
&\le \lambda_{i,m}^{\frac{n-2s}{2}}\msr_m^{2s-n} \int_{B(0,\vep)} \frac{1}{|\ty_{ij,m}-\tom_{ij,m}|^{n-2s}} \frac{\textup{d}\tom_{ij,m}}{\la\tom_{ij,m}+z_{ij,m}\ra^{4s}} \\
&\lesssim \vep^n \lambda_{i,m}^{\frac{n-2s}{2}}\msr_m^{2s-n} \frac{1}{|\ty_{ij,m}|^{n-2s}}
\end{aligned}
\end{equation}
where we used $|\ty_{ij,m}-\tom_{ij,m}| \ge \frac{1}{2}|\ty_{ij,m}|$, which comes from $|\ty_{ij,m}| \ge L'$, to get the last inequality. Notice that $|z_{ij,m}| \le \frac{1}{2}|\ty_{ij,m}|$ and
\[|\ty_{ij,m}|-|z_{ij,m}| \le \lambda_{i,m}|x-z_{i,m}| \le |\ty_{ij,m}|+|z_{ij,m}|,\]
which imply
\[\frac{1}{2}|\ty_{ij,m}| \le |y_{i,m}| = \lambda_{i,m}|x-z_{i,m}| \le \frac{3}{2}|\ty_{ij,m}|.\]
Thus
\begin{align*}
&\ \vep^n \lambda_{i,m}^{\frac{n-2s}{2}}\msr_m^{2s-n} \frac{1}{|\ty_{ij,m}|^{n-2s}} \\
&\lesssim \vep^n \left[(L')^{4s-n} \frac{\lambda_{i,m}^{\frac{n-2s}{2}}\msr_m^{2s-n}}{\la y_{i,m} \ra^{2s}}\mone_{\{|y_{i,m}| < \msr_m\}}
+ (L')^{-2s}\msr_m^{6s-n}\frac{\lambda_{i,m}^{\frac{n-2s}{2}}\msr_m^{-4s}}{|y_{i,m}|^{n-4s}}\mone_{\{|y_{i,m}| \ge \msr_m\}}\right] \\
&\le \vep^n \mcw_m(x).
\end{align*}

\noindent \textit{Case 2:} We assume that $\frac{2\vep}{\lambda_{i,m}} \le |x-z_{j,m}| \le \frac{L'}{\lambda_{i,m}}$ for some $j \in \mcd(i)$.

As in \eqref{eq:claim232}, we compute
\begin{align*}
\mcj_{i,m;\vep}(x) &\le \lambda_{i,m}^{\frac{n-2s}{2}} \msr_m^{2s-n} \int_{B(0,\vep)} \frac{1}{|\ty_{ij,m}-\tom_{ij,m}|^{n-2s}} \frac{\textup{d}\tom_{ij,m}}{\la\tom_{ij,m}+z_{ij,m}\ra^{4s}} \\
&\lesssim \vep^{2s-n} \lambda_{i,m}^{\frac{n-2s}{2}} \msr_m^{2s-n} \int_{B(0,\vep)} \textup{d}\tom_{ij,m}
\lesssim \vep^{2s} \lambda_{i,m}^{\frac{n-2s}{2}} \msr_m^{2s-n}
\end{align*}
where we employed $|\ty_{ij,m}-\tom_{ij,m}| \ge |\ty_{ij,m}|-|\tom_{ij,m}| \ge 2\vep-\vep = \vep$ to obtain the second inequality. Since
\[|y_{i,m}| = \lambda_{i,m}|x-z_{i,m}| \le \lambda_{i,m}(|x-z_{j,m}|+|z_{i,m}-z_{j,m}|) \le L'+C^* < \msr_m\]
for $m \in \N$ large, we see
\[\vep^{2s} \lambda_{i,m}^{\frac{n-2s}{2}} \msr_m^{2s-n} \lesssim \vep^{2s} \left[1+(L'+C^*)^{2s}\right] \frac{\lambda_{i,m}^{\frac{n-2s}{2}} \msr_m^{2s-n}}{\la y_{i,m} \ra^{2s}} \mone_{\{|y_{i,m}| < \msr_m\}} \lesssim \vep^{2s}(L'+C^*)^{2s}\mcw_m(x).\]

\noindent \textit{Case 3:} We assume that $|x-z_{j,m}| \le \frac{2\vep}{\lambda_{i,m}}$ for some $j \in \mcd(i)$.

We calculate
\begin{align*}
\mcj_{i,m;\vep}(x) &\le \int_{B(x, \frac{3\vep}{\lambda_{i,m}})} \frac{1}{|x-\om|^{n-2s}}\frac{\lambda_{i,m}^{\frac{n+2s}{2}}\msr_m^{2s-n}}{\la\lambda_{i,m}(\om-z_{i,m})\ra^{4s}} \mone_{B(z_{i,m},\frac{\msr_m}{\lambda_{i,m}})}(\om)\dom \\
&\le \lambda_{i,m}^{\frac{n+2s}{2}}\msr_m^{2s-n} \int_{B(x, \frac{3\vep}{\lambda_{i,m}})} \frac{1}{|x-\om|^{n-2s}}\dom \lesssim \vep^{2s} \lambda_{i,m}^{\frac{n-2s}{2}} \msr_m^{2s-n}.
\end{align*}
Because
\[|y_{i,m}| \le \lambda_{i,m}(|x-z_{j,m}|+|z_{i,m}-z_{j,m}|)) \le 2\vep+C^* < \msr_m\]
for $m \in \N$ large, we observe
\[\vep^{2s} \lambda_{i,m}^{\frac{n-2s}{2}} \msr_m^{2s-n} \lesssim \vep^{2s}(2\vep+C^*)^{2s}\mcw_m(x).\]

\medskip
From the above analysis for the three cases and \eqref{eq:claim233}--\eqref{eq:claim234}, we find
\begin{equation}\label{eq:claim23}
\begin{aligned}
\mci_{\neck,m;\vep}(x) &\le C(C^*)^{2s}\theta^{-\frac{n-4s}{2s}} \sum_{i=1}^{\nu} \sum_{j \in \mcd(i)} \mcj_{i,m;\vep}(x) + \frac{\zeta}{6} \mcw_m(x) \\
&\le \left[C(C^*)^{2s}\theta^{-\frac{n-4s}{2s}} \vep^{2s} \mcw_m(x) + \frac{\zeta}{6}\right] \mcw_m(x) \le \frac{\zeta}{3} \mcw_m(x)
\end{aligned}
\end{equation}
for all $x \in \R^n$, provided $\vep \in (0,1)$ small and $m \in \N$ large.

\medskip
Now, by inserting \eqref{eq:claim21}, \eqref{eq:claim22}, and \eqref{eq:claim23} into \eqref{eq:claim20}, we obtain \eqref{eq:claim2}. Consequently, the contradictory inequality \eqref{eq:claim0} holds for all $x \in \R^n$ and large $m \in \N$, implying the validity of \eqref{eq:311}. This completes the proof of Proposition \ref{prop:31}.

\section{Quantitative stability estimate for dimension $2s < n < 6s$}\label{sec:dimlow}
Having the spectral inequality \eqref{eq:342} in hand, one may attempt to argue as in \cite{FG} to derive \eqref{eq:main} for $2s < n < 6s$.
Indeed, Aryan pursued this approach in \cite[Section 2]{Ar}, getting the result provided $s \in (0,1)$.

Here, we present an alternative proof of \eqref{eq:main} whose scheme is close to those in the previous sections.
One can use standard integral norms at this time, and computations in the proof are more straightforward than the high-dimensional case $n \ge 6s$.
\begin{defn}\label{defn:normslow}
We redefine the $*$- and $**$-norms as
\[\|\rho\|_* = \|\rho\|_{\dot{H}^s(\R^n)} \quad \text{and} \quad \|h\|_{**} = \|h\|_{L^{2n \over n+2s}(\R^n)}.\]
\end{defn}

As before, the derivation of \eqref{eq:main} is split into three steps.

\medskip \noindent \doublebox{\textsc{Step 1.}} Assume that $2s < n < 6s$. We set $\sigma$, $\rho$, and $\rho_0$ as in Step 1 of Section \ref{sec:dimhigh1}.
\begin{lemma}\label{lemma:52}
There exists a constant $C > 0$ depending only on $n$, $s$, and $\nu$ such that
\begin{equation}\label{eq:510}
\left\|\sigma^p-\sum_{i=1}^{\nu} U_i^p\right\|_{**} \le C\msq
\end{equation}
where $\msq > 0$ is the value in \eqref{eq:Q}.
\end{lemma}
\begin{proof}
By elementary calculus, \eqref{eq:UiUj}, and the condition $n < 6s$, we have
\begin{align*}
\left\|\sigma^p-\sum_{i=1}^{\nu} U_i^p\right\|_{L^{2n \over n+2s}(\R^n)} &\lesssim \sum_{\substack{i, j = 1,\ldots,\nu, \\ i \ne j}} \left\|U_i^{p-1}U_j\right\|_{L^{2n \over n+2s}(\R^n)}
= \sum_{\substack{i, j = 1,\ldots,\nu, \\ i \ne j}} \(\int_{\R^n} U_i^{2(p-1)n \over n+2s} U_j^{2n \over n+2s}\)^{n+2s \over 2n} \\
&\lesssim \msq^{\min\{p-1,1\} \frac{2n}{n+2s} \cdot \frac{n+2s}{2n}} = \msq. \qedhere
\end{align*}
\end{proof}
We next analyze an associated inhomogeneous equation \eqref{eq:lin}.
\begin{prop}\label{prop:51}
If $h \in L^{2n \over n+2s}(\R^n)$ and $f$ satisfies \eqref{eq:lin}, then there exists a constant $C > 0$ depending only on $n$, $s$, and $\nu$ such that
\begin{equation}\label{eq:511}
\|f\|_* \le C\|h\|_{**}.
\end{equation}
\end{prop}
\begin{proof}
Since the condition $f \in \dot{H}^s(\R^n)$ was assumed in \eqref{eq:lin}, we clearly have that $\|f\|_* < \infty$. The proof consists of two substeps.

\medskip \noindent \fbox{\textsc{Substep 1.}} We claim that there is a constant $C > 0$ depending only on $n$, $s$, and $\nu$ such that
\begin{equation}\label{eq:512}
\sum_{i=1}^{\nu} \sum_{a=1}^{n+1} |c_i^a| \le C\(\|h\|_{**} + \msq\|f\|_*\).
\end{equation}

To show it, we test \eqref{eq:lin} with $Z_j^b$ for any fixed $j = 1,\ldots,\nu$ and $b = 1,\ldots,n+1$ and employ \eqref{eq:UiUj}, \eqref{eq:362}, \eqref{eq:310}, and \eqref{eq:Sobolev}. Then we arrive at
\begin{multline*}
c_j^b \int_{\R^n} U[0,1]^{p-1} Z^b[0,1]^2 + \sum_{\substack{i = 1,\ldots,\nu,\\i \ne j}} c_i^a O(q_{ij}) = p \int_{\R^n} \(U_j^{p-1} - \sigma^{p-1}\) fZ_j^b - \int_{\R^n} hZ_j^b \\
= O\(\left\|\sigma^p - \sum_{i=1}^{\nu} U_i^p\right\|_{L^{2n \over n+2s}(\R^n)} \|f\|_{L^{2n \over n-2s}(\R^n)}\) + O\(\|h\|_{L^{2n \over n+2s}(\R^n)}\)
= O\(\msq\|f\|_* + \|h\|_{**}\),
\end{multline*}
which implies \eqref{eq:512}.

\medskip \noindent \fbox{\textsc{Substep 2.}} We assert that \eqref{eq:511} holds.
Suppose not. There exist sequences of small positive numbers $\{\delta'_m\}_{m \in \N}$,
$\delta'_m$-interacting families $\{\{U_{i,m} = U[z_{i,m},\lambda_{i,m}]\}_{i = 1,\ldots,\nu}\}_{m \in \N}$, functions $\{h_m\}_{m \in \N} \subset L^{2n \over n+2s}(\R^n)$ and $\{f_m\}_{m \in \N} \subset \dot{H}^s(\R^n)$,
and numbers $\{c_{i,m}^a\}_{i = 1,\ldots,\nu,\, a = 1,\ldots, n+1,\, m \in \N}$ satisfying \eqref{eq:316}--\eqref{eq:317}. By \eqref{eq:512},
\begin{equation}\label{eq:513}
\sum_{i=1}^{\nu} \sum_{a=1}^{n+1} \left|c_{i,m}^a\right| \to 0 \quad \text{as } m \to \infty.
\end{equation}
Testing \eqref{eq:317} with $f_m$ and using H\"older's inequality, \eqref{eq:Sobolev}, \eqref{eq:316} and \eqref{eq:513}, we obtain
\[p \int_{\R^n} \sigma_m^{p-1}f_m^2 = \|f_m\|_{\dot{H}^s(\R^n)}^2 + O\(\|h_m\|_{L^{2n \over n+2s}(\R^n)} + \sum_{i=1}^{\nu} \sum_{a=1}^{n+1} \left|c_{i,m}^a\right|\) \to 1 \quad \text{as } m \to \infty.\]

On the other hand, the argument in the proof of \eqref{eq:3491} demonstrates a contradictory result
\begin{equation}\label{eq:514}
\lim_{m \to \infty} \int_{\R^n} \sigma_m^{p-1}f_m^2 = 0.
\end{equation}
Indeed, the sequence $\{f_m\}_{m \in \N}$ shares crucial properties of $\{\vrh_m\}_{m \in \N}$ used in the proof of \eqref{eq:3491}:
\begin{itemize}
\item[-] Each $f_m$ solves an inhomogeneous problem \eqref{eq:317} whose right-hand side tends to $0$ in $\dot{H}^{-s}(\R^n)$ as $m \to \infty$.
\item[-] $\|f_m\|_{\dot{H}^s(\R^n)} = 1$ and $f_m \perp Z_{i,m}^a$ in $\dot{H}^s(\R^n)$ for all $m \in \N$, $i = 1,\ldots,\nu$, and $a = 1,\ldots,n+1$.
\end{itemize}
These, combined with Lemma \ref{lemma:nondeg}, yield \eqref{eq:514}. We omit the details.
\end{proof}

A fixed point argument with Lemma \ref{lemma:52} and Proposition \ref{prop:51} leads to the next result; cf. Proposition \ref{prop:33}.
\begin{prop}
Equation \eqref{eq:31} has a solution $\rho_0$ and a family $\{c_i^a\}_{i=1,\ldots,\nu,\ a=1,\ldots,n+1}$ of numbers such that
\begin{equation}\label{eq:53}
\|\rho_0\|_* \le C\msq \quad \text{and} \quad \sum_{i=1}^{\nu} \sum_{a=1}^{n+1} |c_i^a| \le C\msq
\end{equation}
where $C > 0$ depends only on $n$, $s$, and $\nu$ and $\msq > 0$ is the value in \eqref{eq:Q}.
\end{prop}
\noindent The estimate for $c_i^a$'s in \eqref{eq:53} results from \eqref{eq:510}--\eqref{eq:512}, and the fixed point argument.

\medskip \noindent \doublebox{\textsc{Step 2.}} Set $\rho_1 = \rho-\rho_0$. Then it satisfies \eqref{eq:34}.
\begin{prop}\label{prop:54}
There exists a constant $C > 0$ depending only on $n$, $s$, and $\nu$ that
\begin{equation}\label{eq:541}
\|\rho_1\|_* \le C\(\Gamma(u) + \msq^2\)
\end{equation}
where $\Gamma(u) = \|(-\Delta)^su-|u|^{p-1}u\|_{\dot{H}^{-s}(\R^n)}$.
\end{prop}
\begin{proof}
As in the proof of Proposition \ref{prop:34}, one can adapt the argument in the proof of Lemmas 6.2, 6.3, and Proposition 6.4 in \cite{DSW}, which uses the spectral inequality \eqref{eq:342}.

Compared to the high-dimensional case, we have more terms to treat here, because $n < 6s$ implies that $p > 2$ and so
\begin{equation}\label{eq:542}
\begin{cases}
\left||\sigma+\rho_0+\rho_1|^{p-1}(\sigma+\rho_0+\rho_1) - |\sigma+\rho_0|^{p-1}(\sigma+\rho_0) - p|\sigma+\rho_0|^{p-1}\rho_1 \right| \\
\hspace{275pt} \lesssim |\rho_1|^p + |\sigma+\rho_0|^{p-2}|\rho_1|^2;\\
\left|(\sigma+\rho_0)^{p-1} - \sigma^{p-1}\right| \lesssim |\rho_0|^{p-1} + \sigma^{p-2} |\rho_0|;\\
\displaystyle \left|(\sigma+\rho_0)^{p-1} - U_k^{p-1}\right| \lesssim \sum_{\substack{i = 1,\ldots,\nu, \\ i \ne k}} U_i^{p-1} + |\rho_0|^{p-1} + U_k^{p-2} \Bigg(\sum_{\substack{i = 1,\ldots,\nu, \\ i \ne k}} U_i + |\rho_0|\Bigg).
\end{cases}
\end{equation}
Fortunately, the additional terms such as $|\sigma+\rho_0|^{p-2}|\rho_1|^2$ and $\sigma^{p-2} |\rho_0|$ in \eqref{eq:542} can be controlled well,
and the bound \eqref{eq:341} for $\rho_1$ in Proposition \ref{prop:34} keeps unchanged.
\end{proof}
Putting \eqref{eq:53} and \eqref{eq:541} together leads
\begin{equation}\label{eq:55}
\|\rho\|_{\dot{H}^s(\R^n)} \le \|\rho_0\|_* + \|\rho_1\|_* \le C\(\Gamma(u) + \msq\).
\end{equation}

\medskip \noindent \doublebox{\textsc{Step 3.}} Thanks to \eqref{eq:55}, we only need to check that $\msq \lesssim \Gamma(u)$ to establish \eqref{eq:main}.

Since $n < 6s$, it holds that $p > 2$. Hence
\begin{equation}\label{eq:56}
\left||\sigma+\rho|^{p-1}(\sigma+\rho) - \sigma^p - p\, \sigma^{p-1}\rho\right| \lesssim \sigma^{p-2}\rho^2 + |\rho|^p.
\end{equation}
Testing \eqref{eq:30} with $Z_j^{n+1}$ for any fixed $j = 1,\ldots,\nu$, and employing \eqref{eq:56}, \eqref{eq:362}, H\"older's inequality,
\eqref{eq:Sobolev}, \eqref{eq:510}, \eqref{eq:53}, \eqref{eq:541}, \eqref{eq:55}, and $p > 2$, we observe
\begin{equation}\label{eq:564}
\begin{aligned}
&\ \left|\int_{\R^n} \(\sigma^p-\sum_{i=1}^{\nu} U_i^p\)Z_j^{n+1}\right| \\
&\lesssim \int_{\R^n} \(\sigma^p-\sum_{u=1}^{\nu} U_i^p\) |\rho_0| + \int_{\R^n} \sigma^p |\rho_1|
+ \int_{\R^n} \sigma^{p-1}\rho^2 + \int_{\R^n} |\rho|^p \left|Z_j^{n+1}\right| + \Gamma(u) \\
&\lesssim \Gamma(u) + \msq^2 + \|\rho_1\|_{\dot{H}^s(\R^n)} + \|\rho\|_{\dot{H}^s(\R^n)}^2 + \|\rho\|_{\dot{H}^s(\R^n)}^p \lesssim \Gamma(u) + \msq^2.
\end{aligned}
\end{equation}
Moreover, a suitable modification of the proof of \cite[Lemma 2.1]{DSW} gives \eqref{eq:365} provided $n < 6s$. During the derivation of \eqref{eq:365}, we need the estimate
\begin{align*}
\int_{\R^n} U_i^{p-1} U_j U_k &\lesssim \(\int_{\R^n} U_i^{\frac{3}{2}(p-1)} U_j^{\frac{3}{2}}\)^{\frac{1}{3}}
\(\int_{\R^n} U_i^{\frac{3}{2}(p-1)} U_k^{\frac{3}{2}}\)^{\frac{1}{3}} \(\int_{\R^n} U_j^{\frac{3}{2}} U_k^{\frac{3}{2}}\)^{\frac{1}{3}} \\
&\lesssim \sqrt{q_{ij}q_{ik}q_{jk}}\, |\log q_{jk}|^{\frac{1}{3}} \lesssim \msq^{\frac{3}{2}}|\log \msq|^{\frac{1}{3}} = o(\msq),
\end{align*}
which holds for any $n < 6s$ and $i, j, k = 1,\ldots,\nu$ such that $i \ne j$, $j \ne k$, and $i \ne k$.

As a consequence, we deduce \eqref{eq:366} from \eqref{eq:564} and \eqref{eq:365}. The desired inequality $\msq \lesssim \Gamma(u)$ follows from \eqref{eq:366}. This completes the proof of \eqref{eq:main} for $2s < n < 6s$.

To derive the sharpness of \eqref{eq:main}, one can modify the argument in \cite[Section 5.1]{CK}. We skip it.

\appendix
\section{Auxiliary results}\label{sec:app}
\subsection{Non-degeneracy result}
We prove the non-degeneracy of the bubble $U[z,\lambda]$, which is a minor variation of ones in \cite{DdPS,LX}. We believe that it is of practical use elsewhere.
\begin{lemma}\label{lemma:nondeg}
Let $n \in \N$ and $s \in (0,\frac{n}{2})$. If $Z \in \dot{H}^s(\R^n)$ solves
\[(-\Delta)^s Z - p\,U[0,1]^{p-1} Z = 0 \quad \text{in } \R^n,\]
then $Z \in \textnormal{span}\{Z^0[0,1], Z^1[0,1], \ldots, Z^{n+1}[0,1]$\}.
\end{lemma}
\begin{proof}
As in \eqref{eq:320}, it holds that
\begin{equation}\label{eq:Zint}
Z = \Phi_{n,s} \ast \(p\,U[0,1]^{p-1} Z\) \quad \text{in } \R^n.
\end{equation}
It view of \cite[Lemma 5.1]{LX}, it suffices to know that $Z \in L^{\infty}(\R^n)$. Its verification can be done in the following way:
\begin{itemize}
\item[-] If $n \ge 6s$, we apply the iteration process in Substep 1 of the proof of Proposition \ref{prop:31} to \eqref{eq:Zint}.
\item[-] If $2s < n < 6s$, then the HLS inequality and H\"older's inequality imply
\[\hspace{30pt} \|Z\|_{L^{t^*}(\R^n)} \lesssim \left\|\frac{|Z|}{\la\, \cdot\, \ra^{4s}} \ast \frac{1}{|\cdot|^{n-2s}}\right\|_{L^{t^*}(\R^n)}
\lesssim \left\|\frac{|Z|}{\la\, \cdot\, \ra^{4s}}\right\|_{L^{\zeta_2}(\R^n)}
\lesssim \|Z\|_{L^t(\R^n)} \left\|\frac{1}{\la\, \cdot\, \ra^{4s}}\right\|_{L^{\zeta_1}(\R^n)}\]
for $t = \frac{2n}{n-2s}$, $\zeta_1 \in (\frac{n}{2s}, \frac{2n}{6s-n})$, $\zeta_2 = \frac{t\zeta_1}{\zeta_1+t}$, and $t^* = \frac{n\zeta_2}{n-2s\zeta_2} \in (\frac{2n}{n-2s}, \infty)$.
This means that $Z \in L^{\tit}(\R^n)$ for all $\tit \ge \frac{2n}{n-2s}$. From \eqref{eq:318}, we conclude that $Z \in L^{\infty}(\R^n)$. \qedhere
\end{itemize}
\end{proof}

\subsection{Removability of singularity}
We derive a result on the removability of singularities of a solution to an integral equation, which will be used in the proof of Lemma \ref{lemma:3c6}.
\begin{lemma}\label{lemma:reg2}
Suppose that $n \in \N$, $s \in (0,\frac{n}{2})$, $\alpha \in (0,n)$, and $\beta > 2s$.
Given any $N \in \N$, let $\mfy_1, \ldots, \mfy_N$ be distinct points in $\R^n$. If $f$ and $V$ are functions such that
\begin{equation}\label{eq:reg20}
\begin{cases}
f = \Phi_{n,s} \ast (Vf) &\text{in } \R^n \setminus \{\mfy_1, \ldots, \mfy_N\}, \\
\displaystyle |f(y)| \le C\(1 + \sum_{i=1}^N \frac{1}{|y-\mfy_i|^{\alpha}}\),\, |V(y)| \le \frac{C}{\la y \ra^{\beta}} &\text{for } y \in \R^n \setminus \{\mfy_1, \ldots, \mfy_N\}
\end{cases}
\end{equation}
for some $C > 0$, then $f \in L^{\infty}(\R^n)$.
\end{lemma}
\begin{proof}
Let $C^{**} = 1+\max_{i=1,\ldots,N} |\mfy_i|$. By \eqref{eq:reg20}, $f$ is clearly bounded in the set $B(0,4C^{**})^c$.

Suppose that $y \in B(0,4C^{**}) \setminus \{\mfy_1, \ldots, \mfy_N\}$, it holds that
\begin{equation}\label{eq:reg21}
\begin{aligned}
|f(y)| &\lesssim \int_{\R^n} \frac{1}{|y-\om|^{n-2s}} \frac{\dom}{\la \om \ra^{\beta}} + \sum_{i=1}^N \int_{B(0,2C^{**})^c} \frac{1}{|y-\om|^{n-2s}} \frac{\dom}{|\om|^{\alpha+\beta}} \\
&\ + \sum_{i=1}^N \int_{B(0,2C^{**})} \frac{1}{|y-\om|^{n-2s}} \frac{\dom}{|\om-\mfy_i|^{\alpha}} \\
&\lesssim \frac{1+\log(2+|y|) \mone_{\{\beta=n\}}}{\la y \ra^{\min\{\beta,n\}-2s}} + \sum_{i=1}^N \int_{B(0,2C^{**})} \frac{1}{|y-\om|^{n-2s}} \frac{\dom}{|\om-\mfy_i|^{\alpha}}
\end{aligned}
\end{equation}
where the integrals on the first line were computed as in \eqref{eq:315}.

We shall estimate the rightmost integral in \eqref{eq:reg21}.
Fix any non-negative $\zeta \in (\alpha-2s,\min\{\alpha,n-2s\}]$.
By handling the cases $\{\om \in \R^n: |\om-y| < \frac{1}{2}|y-\mfy_i|\}$, $\{\om \in \R^n: |\om-\mfy_i| < \frac{1}{2}|y-\mfy_i|\}$, and $\{\om \in \R^n: \min\{|\om-y|, |\om-\mfy_i|\} \ge \frac{1}{2}|y-\mfy_i|\}$ separately, one can derive
\begin{equation}\label{eq:reg23}
\frac{1}{|y-\om|^{n-2s}} \frac{1}{|\om-\mfy_i|^{\alpha}} \le \frac{c}{|y-\mfy_i|^{\zeta}} \(\frac{1}{|\om-y|^{n-2s+\alpha-\zeta}} + \frac{1}{|\om-\mfy_i|^{n-2s+\alpha-\zeta}}\)
\end{equation}
where $c > 0$ is determined by $n$, $s$, $\alpha$, and $\zeta$. By \eqref{eq:reg23} and the estimate
\begin{align*}
\int_{B(0,2C^{**})} \frac{\dom}{|\om-\mfy_i|^{n-2s+\alpha-\zeta}} &\lesssim \int_{\left\{|\om-\mfy_i|<\frac{|\mfy_i|}{2}\right\}} \frac{\dom}{|\om-\mfy_i|^{n-2s+\alpha-\zeta}} + \frac{1}{|\mfy_i|^{n-2s+\alpha-\zeta}} \int_{\left\{|\om|<\frac{|\mfy_i|}{2}\right\}} \dom \\
&\quad + \int_{\left\{\frac{|\mfy_i|}{2} \le |\om|< 2C^{**}\right\} \cap \left\{|\om-\mfy_i|\ge\frac{|\mfy_i|}{2}\right\}} \frac{\dom}{|\om|^{n-2s+\alpha-\zeta}} \\
&\lesssim |\mfy_i|^{\zeta-(\alpha-2s)} + |\mfy_i|^{\zeta-(\alpha-2s)} + (C^{**})^{\zeta-(\alpha-2s)} \lesssim (C^{**})^{\zeta-(\alpha-2s)},
\end{align*}
we deduce
\[\int_{B(0,2C^{**})} \frac{1}{|y-\om|^{n-2s}} \frac{\dom}{|\om-\mfy_i|^{\alpha}} \lesssim \frac{c (C^{**})^{\zeta-(\alpha-2s)}}{|y-\mfy_i|^{\zeta}}\]
provided $|y| < 4C^{**}$.

Therefore,
\[|f(y)| \le C'\(1 + \sum_{i=1}^N \frac{1}{|y-\mfy_i|^{\zeta}}\) \quad \text{for } y \in \R^n \setminus \{\mfy_1, \ldots, \mfy_N\}\]
where $C' > 0$ is determined by $n$, $s$, $\alpha$, $\beta$, $\zeta$, $C^{**}$, and $C$ in \eqref{eq:reg20}.
Feeding back this information into \eqref{eq:reg21}, we can iterate the above process until we get $f \in L^{\infty}(\R^n)$.
\end{proof}

\section{Technical computations}\label{sec:app2}
Throughout this appendix, we assume that $n > 6s$.

\subsection{Derivation of \eqref{eq:3c35}}\label{subsec:app22}
We recall that
\[f_m(x) = \int_{\R^n} \frac{\ga_{n,s}}{|x-\om|^{n-2s}}\(\sigma_m^{p-1}f_m+h_m + \sum_{i=1}^{\nu} \sum_{a=1}^{n+1} c_{i,m}^a U_{i,m}^{p-1} Z_{i,m}^a\)(\om) \dom \quad \text{for } x \in \R^n.\]
For $s \in (\frac{1}{2},\frac{n}{2})$, we will prove that
\begin{equation}\label{eq:nablafm}
\nabla f_m(x) = (2s-n)\ga_{n,s} \int_{\R^n}\frac{(x-\om)}{|x-\om|^{n-2s+2}} \(\sigma_m^{p-1}f_m + h_m + \sum_{i=1}^{\nu} \sum_{a=1}^{n+1} c_{i,m}^a U_{i,m}^{p-1} Z_{i,m}^a\)(\om) \dom
\end{equation}
for $x \in \R^n$. It suffices to check that the integral on the right-hand side, which we call $g_m(x)$, is well-defined.

To analyze the integral, we decompose $\R^n$ into two subsets $\{|y_{i,m}|\le \frac32 \msr_m\}$ and $\{|y_{i,m}|\ge \frac32\msr_m\}$. Then one can examine
\begin{multline*}
\int_{\R^n}\frac{1}{|x-\om|^{n-2s+1}}U_{i,m}^{p-1}(\om) \(w_{i,m}^{\tin}+w_{i,m}^{\tout}\)(\om) \dom\\
\lesssim \frac{\lambda_{i,m}\(w_{i,m}^{\tin}+w_{i,m}^{\tout}\)(x)}{\la y_{i,m}\ra} \(\frac{1}{\la y_{i,m} \ra^{2s}} \mone_{\{|y_{i,m}| < \msr_m\}} + \frac{\log |y_{i,m}|}{|y_{i,m}|^{2s}} \mone_{\{|y_{i,m}| \ge \msr_m\}}\)
\end{multline*}
and
\[\int_{\R^n}\frac{1}{|x-\om|^{n-2s+1}}\(v_{i,m}^{\tin}+v_{i,m}^{\tout}\)(\om) \dom\lesssim \frac{\lambda_{i,m}}{\la y_{i,m}\ra}\(w_{i,m}^{\tin}+w_{i,m}^{\tout}\)(x).\]
By applying \eqref{eq:upb1}--\eqref{eq:upb4} and \eqref{eq:upb2}--\eqref{eq:upb6}, we observe that for any $M>1$,
\begin{align*}
&\ \int_{\R^n}\frac{1}{|x-\om|^{n-2s+1}}\(\sigma_m^{p-1}f_m+h_m\)(\om)\dom \\
&\lesssim \sum_{i=1}^{\nu}\frac{\lambda_{i,m}\(w_{i,m}^{\tin}+w_{i,m}^{\tout}\)(x)}{\la y_{i,m}\ra}
\left[M^{3n} \(\frac{1}{\la y_{i,m} \ra^{2s}} \mone_{\{|y_{i,m}| < \msr_m\}} + \frac{\log |y_{i,m}|}{|y_{i,m}|^{2s}} \mone_{\{|y_{i,m}| \ge \msr_m\}}\)\right. \\
&\left. \hspace{145pt} + M^{4s}\msr_m^{-2s} + M^{-2s} + \|h_m\|_{**}\right] \\
&\lesssim \sum_{i=1}^{\nu}\frac{\lambda_{i,m}}{\la y_{i,m}\ra}\(w_{i,m}^{\tin}+w_{i,m}^{\tout}\)(x).
\end{align*}
This gives rise to
\[|g_m(x)|\lesssim \sum_{i=1}^{\nu}\frac{\lambda_{i,m}}{\la y_{i,m}\ra}\(w_{i,m}^{\tin}+w_{i,m}^{\tout}\)(x),\]
since \eqref{eq:cjb} implies
\begin{multline*}
\int_{\R^n}\frac{1}{|x-\om|^{n-2s+1}}\left|\sum_{i=1}^{\nu} \sum_{a=1}^{n+1} c_{i,m}^a (U_{i,m}^{p-1} Z_{i,m}^a)(\om)\right| \dom\\
\lesssim\(\|h_m\|_{**} \msr_m^{2s-n} + \|f_m\|_* \msr_m^{-(n+2s)}\) \sum_{i=1}^{\nu}\frac{\lambda_{i,m}}{\la y_{i,m}\ra}U_{i,m}(x).
\end{multline*}
Therefore, \eqref{eq:nablafm} is valid.

Considering the relationship that
\[\big|\nabla \hf_m(y)\big|=\lambda_{i_0,m}^{-1}\mcw_m(x_m)^{-1}|\nabla f_m(x)| \quad \text{for } x=\lambda_{i_0,m}^{-1}y+z_{i_0,m},\]
we find
\begin{align*}
\big|\nabla\hf_m(y)\big| &\lesssim \frac{1}{\lambda_{i_0,m}\mcw_m(x_m)} \sum_{i=1}^{\nu} \frac{\lambda_{i,m}}{\la\lambda_{i,m}\lambda_{i_0,m}^{-1}(y-z_{i_0i,m})\ra} \(w_{i,m}^{\tin}+w_{i,m}^{\tout}\)\big(\lambda_{i_0,m}^{-1}y+z_{i_0,m}\big)\\
&\lesssim L_0^{n-4s} +\sum_{j \prec i_0}\(\frac{\lambda_{j,m}}{\lambda_{i_0,m}}\)^{\frac{n-2s}{2}+1}+ \sum_{i_0\prec j} \left[\frac{L_0^{2s}}{|y-z_{i_0j,\infty}|^{2s+1}} \mone_{\left\{\left|\frac{\lambda_{j,m}}{\lambda_{i_0,m}}(y-z_{i_0j,\infty})\right| < \msr_m\right\}}\right.\\
&\hspace{185pt} \left.+\frac{L_0^{n-4s}}{|y-z_{i_0j,\infty}|^{n-4s+1}} \mone_{\left\{\left|\frac{\lambda_{j,m}}{\lambda_{i_0,m}}(y-z_{i_0j,\infty})\right| \ge \msr_m\right\}}\right]\\
&\ +\sum_{|z_{i_0j,m}| \to \infty} \left[\frac{\lambda_{j,m}\la\lambda_{j,m}(x_m-z_{j,m})\ra^{2s}}{\lambda_{i_0.m} \la\frac{\lambda_{j,m}}{\lambda_{i_0.m}}(y-z_{i_0j,m})\ra^{2s+1}} \mone_{\left\{\left|\frac{\lambda_{j,m}}{\lambda_{i_0,m}}(y-z_{i_0j,m})\right| < \msr_m\right\}}\right.\\
&\hspace{65pt} \left.+\frac{\lambda_{j,m}|\lambda_{j,m}(x_m-z_{j,m})|^{n-4s}}{\lambda_{i_0,m} \left|\frac{\lambda_{j,m}}{\lambda_{i_0.m}}(y-z_{i_0j,m})\right|^{n-4s+1}} \mone_{\left\{\left|\frac{\lambda_{j,m}}{\lambda_{i_0,m}}(y-z_{i_0j,m})\right| \ge \msr_m\right\}}\right] \lesssim 1
\end{align*}
for any ball $B' \subset \mck_l \subset \R^n \setminus \wtmcz_{\infty}$ and $y \in B'$. This concludes the proof of \eqref{eq:3c35}.

\begin{rmk}
The previous argument reveals that
\begin{align*}
\big|\nabla^{\ga} f_m(x)\big| &\lesssim \left|\int_{\R^n}\frac{1}{|x-\om|^{n-2s+\ga}} \(\sigma_m^{p-1}f_m+h_m + \sum_{i=1}^{\nu} \sum_{a=1}^{n+1} c_{i,m}^a U_{i,m}^{p-1} Z_{i,m}^a\)(\om) \dom\right| \\
&\lesssim \sum_{i=1}^{\nu}\(\frac{\lambda_{i,m}}{\la y_{i,m}\ra}\)^{\ga}\(w_{i,m}^{\tin}+w_{i,m}^{\tout}\)(x)
\end{align*}
for any integer $\ga \in [1, \lfloor 2s \rfloor]$, since $n-2s+\ga<n$. Consequently,
\[\sup_{m \in \N} \big\|\hf_m\big\|_{C^{\ga}(\overline{B'})} \lesssim 1.\]
\end{rmk}

\subsection{Derivation of \eqref{eq:3c37}}\label{subsec:app23}
Fix any $y \in B' \subset \mck_l \subset \R^n \setminus \wtmcz_{\infty}$. It follows directly from \eqref{eq:3c32} that
\begin{equation}\label{eq:wmm}
\frac{\mcw_m\big(\lambda_{i_0,m}^{-1}y+z_{i_0,m}\big)}{\mcw_m(x_m)}\lesssim 1.
\end{equation}

Recalling \eqref{mchm} and noting that
\[U_{j,m}\big(\lambda_{i_0,m}^{-1}y+z_{i_0,m}\big) = o(1)U_{i_0,m}\big(\lambda_{i_0,m}^{-1}y+z_{i_0,m}\big) \quad \text{for } j\notin \mcd(i_0) \text{ and } y\in B',\]
we infer from \eqref{eq:cjb}, \eqref{eq:316}, \eqref{eq:325}, and \eqref{eq:324} that
\begin{multline}\label{lihm}
\int_{\R^n} \frac{1}{|y-\om|^{n-2s}} \mch_m(\om) \dom \\
\lesssim \|h_m\|_{**} \frac{\mcw_m\big(\lambda_{i_0,m}^{-1}y+z_{i_0,m}\big)}{\mcw_m(x_m)}+\frac{\sum_{j=1}^{\nu} \sum_{a=1}^{n+1} |c_{j,m}^a|U_{j,m}\big(\lambda_{i_0,m}^{-1}y+z_{i_0,m}\big)}{\mcw_m(x_m)} \to 0
\end{multline}
as $m \to \infty$.

In the following, we will justify
\begin{align}
&\ \lim_{m \to \infty}\int_{\R^n}\frac{1}{|y-\om|^{n-2s}}\left\{\lambda_{i_0,m}^{-{n-2s \over 2}} \sigma_m\big(\lambda_{i_0,m}^{-1}\om+z_{i_0,m}\big)\right\}^{p-1} \hf_m(\om) \dom \nonumber\\
&=\int_{\R^n} \frac{1}{|y-\om|^{n-2s}} \(U[0,1]^{p-1}\hf_{\infty}\)(\om) \dom \label{limeq}
\end{align}
for each fixed $y \in B'$. Indeed, if it is true, \eqref{eq:3c37} will be an immediate consequence of \eqref{eq:3c30}, \eqref{eq:3c36}, and \eqref{lihm}.

\medskip
Given any $M > 4C^*$ large and $\ep>0$ small, we decompose $\R^n$ into
\begin{align*}
\R^n &= \left[B(0,M) \cap \(\cup_{i \in \mcd(i_0)} B(z_{i_0i,\infty},\ep)\)\right] \bigcup \left[B(0,M) \setminus \(\cup_{i \in \mcd(i_0)} B(z_{i_0i,\infty},\ep)\)\right] \bigcup B(0,M)^c\\
&=: \Omega_1 \cup \Omega_2 \cup \Omega_3.
\end{align*}
We set
\begin{align}
I_1 &:= \int_{\R^n}\frac{1}{|y-\om|^{n-2s}} \(U[0,1]^{p-1}\hf_m\)(\om) \dom\ =\int_{\Omega_1}\cdots + \int_{\Omega_2}\cdots + \int_{\Omega_3}\cdots =: I_{11}+I_{12}+I_{13}, \nonumber \\
I_2 &:= \int_{\R^n}\frac{1}{|y-\om|^{n-2s}} \left[\left\{\(\lambda_{i_0,m}^{-{n-2s \over 2}} \sigma_m\big(\lambda_{i_0,m}^{-1}\cdot+z_{i_0,m}\big)\)^{p-1}-U[0,1]^{p-1}\right\} \hf_m\right](\om) \dom \label{i1i2} \\
&= \int_{\Omega_1}\cdots + \int_{\Omega_2}\cdots + \int_{\Omega_3}\cdots =: I_{21}+I_{22}+I_{23}. \nonumber
\end{align}
It is sufficient to analyze integrals $I_{11}, \ldots, I_{23}$ separately.

\medskip \noindent \doublebox{\textsc{Step 1.}} We take $l>\max\{M, \ep^{-1}\}$. Following Claim 1 in the proof of \cite[Lemma 5.1]{DSW}, one can find
\[\lambda_{i_0,m}^{-{n-2s \over 2}} \sigma_m\big(\lambda_{i_0,m}^{-1}\cdot+z_{i_0,m}\big) \to U[0,1] \quad
\text{in } L^{\infty}(\Omega_2) \quad \text{as } m \to \infty\]
by virtue of $\Omega_2 \subset \mck_l$. Moreover, owing to \eqref{eq:3c36}, $\hf_m \to \hf_{\infty}$ uniformly in $\Omega_2$ as $m \to \infty$, so
\begin{equation}\label{i122}
\begin{cases}
\displaystyle I_{12} \to \int_{\Omega_2}\frac{1}{|y-\om|^{n-2s}}\(U[0,1]^{p-1}\hf_\infty\)(\om) \dom, \\
\displaystyle I_{22} \to 0
\end{cases}
\quad \text{as } m \to \infty
\end{equation}
for any fixed $y \in B'$.

For $I_{11}$, we notice from \eqref{eq:mcvw} that
\begin{equation}\label{eq:3c33}
\mcw_m(x_m) \ge \frac{\lambda_{i_0}^{n-2s \over 2} \msr_m^{2s-n}}{\la \lambda_{i_0,m}(x_m-z_{i_0,m}) \ra^{2s}} \gtrsim L_0^{-2s} \lambda_{i_0}^{n-2s \over 2} \msr_m^{2s-n}.
\end{equation}
By using \eqref{eq:usua} and \eqref{eq:3c33}, we see that for $w\in \Omega_1$,
\[\left|\hf_m(\om)\right| \lesssim L_0^{2s} + \frac{1}{\mcw_m(x_m)} \sum_{j \in \mcd(i_0)} \(w_{j,m}^{\tin}+w_{j,m}^{\tout}\)\big(\lambda_{i_0,m}^{-1}\om+z_{i_0,m}\big).\]
Fix any $i \in \mcd(i_0)$. If $|y-z_{i_0i,\infty}|\le 2\ep$, then
\[\int_{B(z_{i_0i,\infty},\ep)} \frac{1}{|y-\om|^{n-2s}}U[0,1]^{p-1}(\om) \dom \lesssim \int_{B(y,3\ep)} \frac{1}{|y-\om|^{n-2s}} \dom \simeq \ep^{2s}.\]
If $|y-z_{i_0i,\infty}| \ge 2\ep$, then
\[\int_{B(z_{i_0i,\infty},\ep)}\frac{1}{|y-\om|^{n-2s}}U[0,1]^{p-1}(\om) \dom \lesssim \ep^{2s-n}\int_{B(0,\ep)} \dom \simeq \ep^{2s}.\]
From \eqref{eq:upb4}, \eqref{eq:upb2}, \eqref{eq:324}, and \eqref{eq:wmm}, we know
\begin{align*}
&\begin{medsize}
\displaystyle \ \frac{1}{\mcw_m(x_m)}\sum_{j \in \mcd(i_0)}\int_{B(z_{i_0j,\infty},\ep)}\frac{1}{|y-\om|^{n-2s}} U[0,1]^{p-1}(\om) \(w_{j,m}^{\tin}+w_{j,m}^{\tout}\)\big(\lambda_{i_0,m}^{-1}\om+z_{i_0,m}\big) \dom
\end{medsize} \\
&\begin{medsize}
\displaystyle \lesssim \frac{1}{\mcw_m(x_m)}\sum_{j \in \mcd(i_0)} \int_{B(z_{i_0j,m},\ep+o(1))} \frac{1}{|y-\om|^{n-2s}} \left[\msr_m^{-2s}+(o(1)+\ep)^{2s}+\(\frac{\lambda_{i_0,m}}{\lambda_{j,m}}\)^{2s}\right]
\end{medsize} \\
&\begin{medsize}
\displaystyle \hspace{180pt} \times \(v_{i_0,m}^{\tin}+v_{i_0,m}^{\tout}+v_{j,m}^{\tin}+v_{j,m}^{\tout}\) \big(\lambda_{i_0,m}^{-1}\om+z_{i_0,m}\big) \dom
\end{medsize} \\
&\begin{medsize}
\displaystyle \lesssim \frac{1}{\mcw_m(x_m)} \sum_{j \in \mcd(i_0)} \left[\msr_m^{-2s} + (o(1)+\ep)^{2s} + \(\frac{\lambda_{i_0,m}}{\lambda_{j,m}}\)^{2s}\right] \(w_{j,m}^{\tin}+w_{j,m}^{\tout}\)\big(\lambda_{i_0,m}^{-1}y+z_{i_0,m}\big)
\end{medsize} \\
&\begin{medsize}
\displaystyle \lesssim \sum_{j \in \mcd(i_0)} \left[\msr_m^{-2s}+(o(1)+\ep)^{2s}+\(\frac{\lambda_{i_0,m}}{\lambda_{j,m}}\)^{2s}\right] \simeq \ep^{2s}+o(1).
\end{medsize}
\end{align*}
In addition,
\begin{align*}
&\ \sum_{\substack{i,j \in \mcd(i_0)\\ z_{i_0i,\infty} \ne z_{i_0j,\infty}}} \int_{B(z_{i_0i,\infty},\ep)} \frac{1}{|y-\om|^{n-2s}}U[0,1]^{p-1}(\om) \frac{1}{\mcw_m(x_m)} \(w_{j,m}^{\tin}+w_{j,m}^{\tout}\)\big(\lambda_{i_0,m}^{-1}\om+z_{i_0,m}\big) \dom\\
&\lesssim \sum_{\substack{i,j \in \mcd(i_0),\\ z_{i_0i,\infty}\ne z_{i_0j,\infty}}}\int_{B(z_{i_0i,\infty},\ep)} \frac{1}{|y-\om|^{n-2s}}U[0,1]^{p-1}(\om) \left[\frac{1}{|\om-z_{i_0j,\infty}|^{2s}}+\frac{1}{|\om-z_{i_0j,\infty}|^{n-4s}}\right] \dom\\
&\lesssim \sum_{i \in \mcd(i_0)} \int_{B(z_{i_0i,\infty},\ep)} \frac{1}{|y-\om|^{n-2s}}U[0,1]^{p-1}(\om) \dom \lesssim \ep^{2s}
\end{align*}
where we choose $\ep$ so small that $\ep <\frac12|z_{i_0i,\infty}-z_{i_0j,\infty}|$ for the second inequality. Therefore,
\begin{equation}\label{i21}
|I_{11}|\lesssim \ep^{2s}+o(1).
\end{equation}

Let us turn to $I_{13}$. For $w\in\Omega_3$, it holds that
\begin{multline}\label{hfmp3}
|\hf_m(\om)|\lesssim \(L_0^{2s}M^{-2s}+L_0^{n-4s}M^{-(n-4s)}\) + \sum_{j \prec i_0}\(\frac{\lambda_{j,m}}{\lambda_{i_0,m}}\)^{2s}\\
+ \frac{1}{\mcw_m(x_m)} \sum_{|z_{i_0j,m}| \to \infty} \(w_{j,m}^{\tin}+w_{j,m}^{\tout}\)\big(\lambda_{i_0,m}^{-1}\om+z_{i_0,m}\big).
\end{multline}
On one hand, we have
\[\int_{\Omega_3}\frac{1}{|y-\om|^{n-2s}}U[0,1]^{p-1}(\om) \dom \lesssim \frac{1}{1+|y|^{2s}} \lesssim 1.\]
On the other hand, by \eqref{eq:upb1}-\eqref{eq:upb5}, \eqref{eq:324}, and \eqref{eq:wmm},
\begin{align*}
&\begin{medsize}
\displaystyle \ \frac{1}{\mcw_m(x_m)}\sum_{|z_{i_0j,m}| \to \infty} \int_{\Omega_3} \frac{1}{|y-\om|^{n-2s}}U[0,1]^{p-1}\(w_{j,m}^{\tin}+w_{j,m}^{\tout}\) \big(\lambda_{i_0,m}^{-1}\om+z_{i_0,m}\big) \dom
\end{medsize} \\
&\begin{medsize}
\displaystyle \lesssim \frac{1}{\mcw_m(x_m)} \sum_{|z_{i_0j,m}| \to \infty} \int_{\Omega_3} \frac{\msr_m^{-2s}+\la z_{i_0j,m}\ra^{-2s}}{|y-\om|^{n-2s}} \(v_{i_0,m}^{\tin}+v_{i_0,m}^{\tout}+v_{j,m}^{\tin}+v_{j,m}^{\tout}\) \big(\lambda_{i_0,m}^{-1}\om+z_{i_0,m}\big) \dom
\end{medsize} \\
&\begin{medsize}
\displaystyle \lesssim \(\msr_m^{-2s}+\sum_{|z_{i_0j,m}| \to \infty}\la z_{i_0j,m}\ra^{-2s}\) \frac{1}{\mcw_m(x_m)} \left[w_{i_0,m}^{\tin}+w_{i_0,m}^{\tout} + \sum_{|z_{i_0j,m}| \to \infty} \(w_{j,m}^{\tin}+w_{j,m}^{\tout}\)\right] \big(\lambda_{i_0,m}^{-1}y+z_{i_0,m}\big)
\end{medsize} \\
&\begin{medsize}
\displaystyle \lesssim L_0^{2s}\(\msr_m^{-2s}+\sum_{|z_{i_0j,m}| \to \infty} \la z_{i_0j,m}\ra^{-2s}\) = o(1).
\end{medsize}
\end{align*}
Summing up, we discover
\begin{equation}\label{i23}
|I_{13}| \lesssim L_0^{2s}M^{-2s}+L_0^{n-4s}M^{-(n-4s)}+o(1).
\end{equation}

Furthermore, by employing
\[\hf_m \to\hf_{\infty} \quad \text{in } \R^n\setminus \wtmcz_{\infty} \quad \text{as } m \to \infty.\]
and Fatou's Lemma, we deduce
\begin{align}
\int_{\Omega_1}\frac{1}{|y-\om|^{n-2s}}\(U[0,1]^{p-1}|\hf_{\infty}|\)(\om) \dom &\le \liminf_{m \to \infty} \int_{\Omega_1} \frac{1}{|y-\om|^{n-2s}}\(U[0,1]^{p-1}|\hf_m|\)(\om) \dom \nonumber \\
&\lesssim \ep^{2s} \label{hfi1}
\end{align}
and
\begin{equation}\label{hfi2}
\begin{aligned}
&\begin{medsize}
\displaystyle \ \int_{\Omega_3} \frac{1}{|y-\om|^{n-2s}}\(U[0,1]^{p-1}|\hf_{\infty}|\)(\om) \dom
\end{medsize} \\
&\begin{medsize}
\displaystyle \le \liminf_{m \to \infty} \int_{\Omega_3} \frac{1}{|y-\om|^{n-2s}}\(U[0,1]^{p-1}|\hf_m|\)(\om) \dom
\end{medsize} \\
&\begin{medsize}
\displaystyle \lesssim \liminf_{m \to \infty} \left[L_0^{2s}M^{-2s}+L_0^{n-4s}M^{-(n-4s)} + \sum_{j \prec i_0} \(\frac{\lambda_{j,m}}{\lambda_{i_0,m}}\)^{2s} + L_0^{2s}\(\msr_m^{-2s} + \sum_{|z_{i_0j,m}| \to \infty} \la z_{i_0j,m}\ra^{-2s}\)\right]
\end{medsize} \\
&\begin{medsize}
\displaystyle = L_0^{2s}M^{-2s}+L_0^{n-4s}M^{-(n-4s)}.
\end{medsize}
\end{aligned}
\end{equation}

By collecting \eqref{i122}, \eqref{i21}, and \eqref{i23}-\eqref{hfi2}, we conclude
\begin{multline}\label{api2}
I_1 = \int_{\R^n}\frac{1}{|y-\om|^{n-2s}}(U[0,1]^{p-1}\hf_{\infty})(\om) \dom+o(1)\\
+ O\(L_0^{2s}M^{-2s}+L_0^{n-4s}M^{-(n-4s)}+\ep^{2s}\).
\end{multline}

\medskip \noindent \doublebox{\textsc{Step 2.}} We next estimate $I_2$. Direct computations similar to the previous step yield
\begin{multline*}
\begin{medsize}
\displaystyle \ \frac{1}{\mcw_m(x_m)} \int_{\Omega_1} \frac{1}{|y-\om|^{n-2s}} \left[\(\sum_{i \prec i_0}\lambda_{i_0,m}^{-\frac{n-2s}{2}} U_{i,m}\)^{p-1}\(w_{i_0,m}^{\tin}+w_{i_0,m}^{\tout}\)\right] \big(\lambda_{i_0,m}^{-1}\om+z_{i_0,m}\big) \dom
\end{medsize} \\
\begin{medsize}
\hspace{200pt} \displaystyle \lesssim \sum_{i \prec i_0} \(\frac{\lambda_{i,m}}{\lambda_{i_0,m}}\)^{2s} \int_{\Omega_1} \frac{1}{|y-\om|^{n-2s}} \dom \lesssim \sum_{i \prec i_0}\(\frac{\lambda_{i,m}}{\lambda_{i_0,m}}\)^{2s}\ep^{2s}
\end{medsize} \\
\begin{medsize}
\displaystyle \(\text{since } \(w_{i_0,m}^{\tin}+w_{i_0,m}^{\tout}\)\big(\lambda_{i_0,m}^{-1}\om+z_{i_0,m}\big) \lesssim \mcw_m(x_m)\),
\end{medsize}
\end{multline*}
\begin{multline*}
\begin{medsize}
\displaystyle \ \frac{1}{\mcw_m(x_m)} \int_{\Omega_1} \frac{1}{|y-\om|^{n-2s}} \left[\(\sum_{i_0 \prec i}\lambda_{i_0,m}^{-\frac{n-2s}{2}} U_{i,m}\)^{p-1}\(w_{i_0,m}^{\tin}+w_{i_0,m}^{\tout}\)\right] \big(\lambda_{i_0,m}^{-1}\om+z_{i_0,m}\big) \dom
\end{medsize} \\
\begin{medsize}
\displaystyle \lesssim \msr_m^{-2s} = o(1) \quad (\text{by \eqref{eq:upb1}, \eqref{eq:upb3}, \eqref{eq:324}, and \eqref{eq:wmm}}),
\end{medsize}
\end{multline*}
\begin{multline*}
\begin{medsize}
\displaystyle \ \frac{1}{\mcw_m(x_m)} \int_{\Omega_1} \frac{1}{|y-\om|^{n-2s}} \left[\(\sum_{|z_{i_0i,m}| \to \infty} \lambda_{i_0,m}^{-\frac{n-2s}{2}} U_{i,m}\)^{p-1}\(w_{i_0,m}^{\tin}+w_{i_0,m}^{\tout}\)\right] \big(\lambda_{i_0,m}^{-1}\om+z_{i_0,m}\big) \dom
\end{medsize} \\
\begin{medsize}
\displaystyle \lesssim \msr_m^{-2s}+\sum_{|z_{i_0i,m}| \to \infty}\la z_{i_0j,m}\ra^{-2s} = o(1)
\quad (\text{by \eqref{eq:upb1}-\eqref{eq:upb5}, \eqref{eq:324}, and \eqref{eq:wmm}}),
\end{medsize}
\end{multline*}
and
\begin{align*}
&\begin{medsize}
\displaystyle \ \frac{1}{\mcw_m(x_m)} \int_{\Omega_3} \frac{1}{|y-\om|^{n-2s}} \left[\(\sum_{j \ne i_0} \lambda_{i_0,m}^{-\frac{n-2s}{2}} U_{j,m}\)^{p-1} \(w_{i_0,m}^{\tin}+w_{i_0,m}^{\tout}\)\right] \big(\lambda_{i_0,m}^{-1}\om+z_{i_0,m}\big) \dom
\end{medsize} \\
&\begin{medsize}
\displaystyle \lesssim \(L_0^{2s}M^{-2s}+L_0^{n-4s}M^{-(n-4s)}\) \int_{\R^n} \frac{1}{|y-\om|^{n-2s}} \(\sum_{j \ne i_0} \lambda_{i_0,m}^{-\frac{n-2s}{2}} U_{j,m}\)^{p-1} \big(\lambda_{i_0,m}^{-1}\om+z_{i_0,m}\big) \dom
\end{medsize} \\
&\begin{medsize}
\displaystyle \lesssim \(L_0^{2s}M^{-2s}+L_0^{n-4s}M^{-(n-4s)}\) \sum_{j \ne i_0} \frac{1}{\la\frac{\lambda_{j,m}}{\lambda_{i_0,m}}(y-z_{i_0j,m})\ra^{2s}} \lesssim L_0^{2s}M^{-2s}+L_0^{n-4s}M^{-(n-4s)}.
\end{medsize}
\end{align*}

By taking account of \eqref{eq:upb1}-\eqref{eq:upb4} and \eqref{eq:upb2}-\eqref{eq:upb6}, keeping \eqref{eq:usua} and \eqref{eq:wmm} in mind, and arguing as in \eqref{eq:323}, we calculate
\begin{align*}
&\begin{medsize}
\displaystyle \ \frac{1}{\mcw_m(x_m)} \int_{\Omega_1 \cup \Omega_3} \frac{1}{|y-\om|^{n-2s}} \left[\(\sum_{j\ne i_0}\lambda_{i_0,m}^{-\frac{n-2s}{2}}U[z_{i,m},\lambda_{j,m}]\)^{p-1} \sum_{i \ne i_0} \(w_{i,m}^{\tin}+w_{i,m}^{\tout}\)\right] \big(\lambda_{i_0,m}^{-1}\om+z_{i_0,m}\big) \dom
\end{medsize} \\
&\begin{medsize}
\displaystyle \lesssim \frac{1}{\mcw_m(x_m)} \sum_{i \ne i_0} \(w_{i,m}^{\tin}+w_{i,m}^{\tout}\) \big(\lambda_{i_0,m}^{-1}y+z_{i_0,m}\big) \cdot \left[M_0^{4s}\msr_m^{-2s} + M_0^{-2s} \right.
\end{medsize} \\
&\begin{medsize}
\displaystyle \qquad \left. + M_0^{3n} \(\frac{1}{\la \frac{\lambda_{i,m}}{\lambda_{i_0,m}}(y-z_{i_0i,m}) \ra^{2s}} \mone_{\left\{\left|\frac{\lambda_{i,m}}{\lambda_{i_0,m}}(y-z_{i_0i,m})\right| < \msr_m\right\}}
+ \frac{\log \left| \frac{\lambda_{i,m}}{\lambda_{i_0,m}}(y-z_{i_0i,m}) \right|}{\left| \frac{\lambda_{i,m}}{\lambda_{i_0,m}}(y-z_{i_0i,m}) \right|^{2s}} \mone_{\left\{\left|\frac{\lambda_{i,m}}{\lambda_{i_0,m}}(y-z_{i_0i,m})\right| \ge \msr_m\right\}}\)\right]
\end{medsize} \\
&\begin{medsize}
\displaystyle \lesssim \(M_0^{4s}\msr_m^{-2s} + M_0^{-2s}\) + \frac{o(1)M_0^{3n}}{\mcw_m(x_m)} \(w_{i_0,m}^{\tin}+w_{i_0,m}^{\tout}\)\big(\lambda_{i_0,m}^{-1}y+z_{i_0,m}\big)
\end{medsize} \\
&\begin{medsize}
\displaystyle \ + \sum_{i \in \mcd(i_0)} M_0^{3n} \left[\frac{1}{\la \frac{\lambda_{i,m}}{\lambda_{i_0,m}}(y-z_{i_0i,m}) \ra^{2s}} \mone_{\left\{\left|\frac{\lambda_{i,m}}{\lambda_{i_0,m}}(y-z_{i_0i,m})\right| < \msr_m\right\}}
+ \frac{\log \left| \frac{\lambda_{i,m}}{\lambda_{i_0,m}}(y-z_{i_0i,m}) \right|}{\left| \frac{\lambda_{i,m}}{\lambda_{i_0,m}}(y-z_{i_0i,m}) \right|^{2s}} \mone_{\left\{\left|\frac{\lambda_{i,m}}{\lambda_{i_0,m}}(y-z_{i_0i,m})\right| \ge \msr_m\right\}}\right]
\end{medsize} \\
&\begin{medsize}
\displaystyle \lesssim \(M_0^{4s}\msr_m^{-2s} + M_0^{-2s}\) + o(1)M_0^{3n} +
M_0^{3n} \sum_{i \in \mcd(i_0)} \left[L_0^{2s}\(\frac{\lambda_{i,m}}{\lambda_{i_0,m}}\)^{-2s} + L_0^{n-4s} \(\frac{\lambda_{i,m}}{\lambda_{i_0,m}}\)^{-2s} \left|\log \frac{\lambda_{i,m}}{\lambda_{i_0,m}}\right|\right].
\end{medsize}
\end{align*}
for $y \in B'$. Hence, with \eqref{i122}, we have proven that
\begin{equation}\label{api1}
|I_2| \lesssim o(1)\ep^{2s}+M_0^{-2s}+L_0^{2s}M^{-2s}+L_0^{n-4s}M^{-(n-4s)}+o(1)M_0^{3n}.
\end{equation}

\medskip
To conclude the derivation of \eqref{limeq}, we gather \eqref{i1i2}, \eqref{api2}, and \eqref{api1}, and take $M, M_0>0$ sufficiently large and $\ep > 0$ small.

\subsection{Derivation of \eqref{eq:last}}\label{apsu3}
We will follow the strategy in Appendix \ref{subsec:app23}.

\medskip
First, the fact that $\hf_m \to \hf_{\infty}$ uniformly in $\Omega_2$ as $m \to \infty$ implies
\[\int_{\Omega_2} U[0,1]^{p-1} Z^a[0,1] \hf_m \dy \to \int_{\Omega_2} U[0,1]^{p-1} Z^a[0,1] \hf_\infty \dy \quad \text{as } m \to \infty.\]
Also, in light of \eqref{eq:3c32}, we find that
\begin{align*}
&\ \int_{\Omega_1} \left|U[0,1]^{p-1} Z^a[0,1] \hf_m\right| \dy \\
&\lesssim \int_{\Omega_1} U[0,1]^{p} \Bigg[L_0^{2s}+\sum_{j \in \mcd(i_0)} \(\frac{L^{2s}_0}{|y-z_{i_0j,\infty}|^{2s}}+\frac{L^{n-4s}_0}{|y-z_{i_0j,\infty}|^{n-4s}}\)\Bigg] \dy\\
&\lesssim \int_{\Omega_1}dy + \sum_{j \in \mcd(i_0)} \int_{B(z_{i_0j,\infty},\ep)} \left[\frac{1}{|y-z_{i_0j,\infty}|^{2s}}+\frac{1}{|y-z_{i_0j,\infty}|^{n-4s}}\right] \dy\\
&\ +\sum_{\substack{i,j \in \mcd(i_0),\\ z_{i_0i,\infty} \ne z_{i_0j,\infty}}} \int_{B(z_{i_0i,\infty},\ep)} \left[\frac{1}{|y-z_{i_0j,\infty}|^{2s}}+\frac{1}{|y-z_{i_0j,\infty}|^{n-4s}}\right] \dy \lesssim \ep^{2s}.
\end{align*}
Next, by carrying out computations with the help of \eqref{hfmp3}, we obtain
\begin{align*}
&\ \int_{\Omega_3} \left|U[0,1]^{p-1} Z^a[0,1] \hf_m \right|\dy \\
&\lesssim \int_{\Omega_3} U[0,1]^{p}(y) \Bigg[\(L_0^{2s}M^{-2s}+L_0^{n-4s}M^{-(n-4s)}\) +\sum_{j \prec i_0}\(\frac{\lambda_{j,m}}{\lambda_{i_0,m}}\)^{2s} \\
&\hspace{90pt} + \sum_{|z_{i_0j,m}| \to \infty} \(\frac{|z_{i_0j,m}|^{2s}}{|y-z_{i_0j,m}|^{2s}} + \frac{|z_{i_0j,m}|^{n-4s}}{|y-z_{i_0j,m}|^{n-4s}}\)\Bigg] \dy\\
&\lesssim \Bigg[\(L_0^{2s}M^{-2s}+L_0^{n-4s}M^{-(n-4s)}\) + \sum_{j \prec i_0} \(\frac{\lambda_{j,m}}{\lambda_{i_0,m}}\)^{2s}\Bigg] M^{-2s}\\
&\ + \sum_{|z_{i_0j,m}| \to \infty} \int_{B\(z_{i_0j,m},\frac{|z_{i_0j,m}|}{2}\)} \left[\frac{1}{|z_{i_0j,m}|^{n}}\frac{1}{|y-z_{i_0j,m}|^{2s}}+\frac{1}{|z_{i_0j,m}|^{6s}}\frac{1}{|y-z_{i_0j,m}|^{n-4s}} \right] \dy \\
&\ +\sum_{|z_{i_0j,m}| \to \infty}\int_{B\(z_{i_0j,m},\frac{|z_{i_0j,m}|}{2}\)^c \cap \Omega_3} U[0,1]^{p}(y) \left[\frac{|z_{i_0j,m}|^{2s}}{|y-z_{i_0j,m}|^{2s}} + \frac{|z_{i_0j,m}|^{n-4s}}{|y-z_{i_0j,m}|^{n-4s}}\right] \dy\\
&\lesssim \Bigg[\(L_0^{2s}M^{-2s}+L_0^{n-4s}M^{-(n-4s)}\) + \sum_{j \prec i_0}\(\frac{\lambda_{j,m}}{\lambda_{i_0,m}}\)^{2s}\Bigg]M^{-2s} + \sum_{|z_{i_0j,m}| \to \infty} |z_{i_0j,m}|^{-2s}+M^{-2s}.
\end{align*}
By making use of Fatou's Lemma, we easily get
\[\int_{\Omega_1} \left|U[0,1]^{p-1} Z^a[0,1] \hf_\infty\right| \dy \le \liminf_{m \to \infty} \int_{\Omega_1} \left|U[0,1]^{p-1} Z^a[0,1] \hf_m\right| \dy \lesssim \ep^{2s}\]
and
\begin{align*}
\int_{\Omega_3} \left|U[0,1]^{p-1} Z^a[0,1] \hf_\infty\right| \dy&\le \liminf_{m \to \infty} \int_{\Omega_3} \left|U[0,1]^{p-1} Z^a[0,1] \hf_m \right|\dy\\
&\lesssim \(L_0^{2s}M^{-2s}+L_0^{n-4s}M^{-(n-4s)}\)M^{-2s}+M^{-2s}.
\end{align*}
All the information above and the second equality of \eqref{eq:3c30} present
\[0=\lim_{m \to \infty}\int_{\R^n} U[0,1]^{p-1} Z^a[0,1] \hf_m \dy = \int_{\R^n} U[0,1]^{p-1} Z^a[0,1] \hf_\infty \dy.\]
Thus, the proof of \eqref{eq:last} is completed.

\end{document}